\documentclass[12pt]{amsart}
\usepackage{hyperref}
\usepackage{amsfonts}
\usepackage{dsfont}
\usepackage{mathrsfs} 
\usepackage{amssymb, changes}
\usepackage{amsxtra}
\usepackage{amstext} 
\usepackage[english]{babel}
\usepackage{tcolorbox}
\usepackage{xcolor}
\usepackage{cleveref}
\usepackage[dvipsnames]{xcolor}
\usepackage{tikz}
\usetikzlibrary{positioning}
\usepackage{xcolor}
\usepackage[framemethod=TikZ]{mdframed}
\newmdtheoremenv[linecolor=gray!40!white,
backgroundcolor=gray!9!white, linewidth=2pt,
topline=false,
rightline=false,
leftline=false]{theorem}{Theorem}[section]
\newmdtheoremenv[linecolor=gray!40!white,
backgroundcolor=gray!9!white, linewidth=2pt,
topline=false,
rightline=false,
leftline=false]{assumption}[theorem]{Assumption}
\newmdtheoremenv[linecolor=gray!40!white, backgroundcolor=gray!9!white, linewidth=2pt,topline=false,rightline=false,leftline=false]{lemma}[theorem]{Lemma}
\newmdtheoremenv[linecolor=gray!40!white, backgroundcolor=gray!9!white, linewidth=2pt,topline=false,rightline=false,leftline=false]{corollary}[theorem]{Corollary}
\newmdtheoremenv[linecolor=gray!40!white, backgroundcolor=gray!9!white, linewidth=2pt,topline=false,rightline=false,leftline=false]{proposition}[theorem]{Proposition}
\newmdtheoremenv[linecolor=gray!40!white, backgroundcolor=gray!9!white, linewidth=2pt,topline=false,rightline=false,leftline=false]{definition}{Definition}

\newtheorem{remark}[theorem]{Remark}

\newcommand{\E}{\mathbb{E}}
\renewcommand{\P}{\mathbb{P}}

\def\Tcal{\mathcal T}

\def\Pcal{{\mathcal P}}
\newcommand{\X}{\mathbf{X }}

\newcommand{\Z}{\mathbb{Z}}
\newcommand{\N}{\mathbb{N}}
\def\eps{\varepsilon}
\newcommand{\rmd}{\,\mathrm{d}}

\def\1{\mathds{1}}
\def\1vec{\mathbf{1}}
\def\0{\mathbf{0}}

\definecolor{VioletRed}{RGB}{0,128,0} 

\begin{document}
\title{Once-excited random walks on general trees}
\author[]{Duy-Bao Le}
\address[D.-B. L.]{Faculty of Mathematics and Applications, Saigon University, 273 An Duong Vuong, Cho Quan, Ho Chi Minh City, Vietnam}
\email{duybaole2004@gmail.com}


\author[]{Tuan-Minh Nguyen}
\address[T.-M. N.]{School of Mathematical Sciences, Monash University, 9 Rainforest Walk, Clayton 3800, Victoria, Australia}
\email{tuanminh.nguyen@monash.edu}
\begin{abstract}
We study once-excited random walks on general trees, modeled by placing a single “cookie” at each vertex. Each cookie acts as a metaphorical reward that is consumed upon the first visit to the vertex where it is placed. At that initial visit, the walk is in an excited mode and randomly performs a nearest-neighbor step according to a specified bias parameter. For all subsequent visits to this vertex, the walk reverts to a non-excited mode and behaves as a symmetric random walk. We consider a random environment on the tree in which the bias parameters are independent random variables. We prove that the process in this random environment exhibits a sharp phase transition between transience and recurrence, where the critical threshold is determined by the branching-ruin number of the tree.
\end{abstract}


\maketitle
\tableofcontents

\section{Introduction}
Excited random walks, also known as cookie random walks, model random movement with long memory, where the probability of moving from a vertex depends on how many times that vertex has been visited previously. This is conceptualized by placing a number of cookies at each vertex, with one cookie consumed each time the vertex is visited. As long as cookies remain at a vertex, the walk operates in an ``excited" mode, resembling a biased random walk. Once all cookies at a vertex are consumed, the walk switches to an unbiased (symmetric) random walk. These processes are non-Markovian, as their transition rules depend on the walk’s local time at the current vertex.

We now formally describe a generalized model where we place one cookie at each vertex. Let $\Tcal=(V,E)$ be a locally finite infinite tree with root $\rho$. Let $\vec{E}=\{(u,v) : \{u,v\}\in E\}$ be the set of all induced oriented edges. For two
adjacent vertices $u$ and $v$, we write $u\sim v$. For a vertex $v\neq\rho$, we denote its parent by $v^{-1}$. Fix two collections of positive real numbers $\boldsymbol{\lambda}=(\lambda_v)_{v\in V}$ and $\boldsymbol{\mu}=(\mu_v)_{v\in V}$ where we set $\lambda_{\rho}=\mu_{\rho}=1$. Let $p_{\boldsymbol{\lambda},\boldsymbol{\mu}}: \vec{E}\times \N\to (0,1)$ be the function defined as follows:
\begin{itemize}
    \item For each $u\sim \rho$ and $n\ge 1$,  $p_{\boldsymbol{\lambda},\boldsymbol{\mu}}(\rho,u,n)=1/\deg(\rho)$;  
    \item For $v\in V\setminus\{\rho\}$, $$p_{\boldsymbol{\lambda},\boldsymbol{\mu}}(v,v^{-1},1)=\frac{\lambda_v}{\lambda_v+\deg(v)-1} \text{ and } p_{\boldsymbol{\lambda},\boldsymbol{\mu}}(v,u,1)=\frac{1}{\lambda_v+\deg(v)-1},$$  for each child $u$ of $v$;
     \item For $n\ge 2$ and $v\in V\setminus\{\rho\}$, $$p_{\boldsymbol{\lambda},\boldsymbol{\mu}}(v,v^{-1},n)=\frac{\mu_v}{\mu_v+\deg(v)-1} \text{ and } p_{\boldsymbol{\lambda},\boldsymbol{\mu}}(v,u,n)=\frac{1}{\mu_v+\deg(v)-1},$$  for each child $u$ of $v$.
\end{itemize}
We consider a non-Markovian random walk $\X=(X_n)_{n\ge0}$ on $\Tcal$ with $X_0=\rho$, whose conditional transition law is given by
\begin{align}\label{P.quenched}
    \P_{\boldsymbol{\lambda},\boldsymbol{\mu}}(X_{n+1}=y|X_n, X_{n-1},\cdots, X_0)=\mathbf{1}_{\{y\sim X_n\}}p_{\boldsymbol{\lambda},\boldsymbol{\mu}}(X_n,y, Z_n(X_n))
\end{align}
where $$Z_n(x):=\sum_{k=0}^{n} \mathbf{1}_{\{X_k=x\}}$$ is the number of visits to $x$ up to time $n$. In other words, during the first visit to a vertex $v$ (i.e., the excited regime), the walk performs a step with \textbf{bias} (toward its parent vertex)  $\lambda_v=\frac{p(v,v^{-1},1)}{p(v,u,1)}$. For subsequent visits to this vertex (i.e., the non-excited regime), it switches to  bias $\mu_v=\frac{p(v,v^{-1},n)}{p(v,u,n)}$ with $n\ge 2$ .

When  $\mu_v=1$ for all $v\in V$, i.e., no bias in the non-excited regime, the process is known in the literature as a \textbf{once-excited random walk} (OERW), denoted by ${\rm OERW}(\boldsymbol{\lambda})$. For general values of $\boldsymbol{\lambda}$ and $\boldsymbol{\mu}$,  we call this process a \textbf{generalized once-excited random walk} (GOERW), and denote it by ${\rm GOERW}(\boldsymbol{\lambda}, \boldsymbol{\mu})$. 

Assume that the bias parameters $\boldsymbol{\lambda}=(\lambda_v)_{v\in V}$ and $\boldsymbol{\mu}=(\mu_v)_{v\in V}$ are collections of random variables governed by a probability $\mathbf P$. These random variables induce a random environment on tree $\Tcal$. Let $\mathbf{E}$ be the expectation associated with $\mathbf{P}$. For each realization of $\boldsymbol{\lambda}$ and $\boldsymbol{\mu}$, the probabilistic measure $\mathbb{P}_{\boldsymbol{\lambda},\boldsymbol{\mu}}$ given by \eqref{P.quenched} is called the \textbf{quenched law} of $\X$. The probability measure
$$\P(\cdot):=\mathbf{E}[\P_{\boldsymbol{\lambda}, \boldsymbol{\mu}}(\cdot)]
$$
is called the \textbf{annealed law} of $\X$. We say that $\X$ \begin{itemize}
    \item is \textbf{recurrent} if the process visits each vertex in $\Tcal$ infinitely often,
    \item is \textbf{transient} if the process visits each vertex in $\Tcal$ finitely often.
\end{itemize}

The concept of once-excited random walks was first introduced by Benjamini and Wilson \cite{BW2003} and later extended to multi-excited random walks by Zerner \cite{Z2005}. In recent years, significant attention has been devoted to the study of one-dimensional nearest-neighbor excited random walks, leading to remarkable phase transition results for the asymptotic behavior of the model, including criteria for recurrence/transience \cite{ABO, Z2005} and non-ballisticity/ballisticity \cite{BasS1, BasS2}, as well as the characterization of the limit distribution in specific regimes — see \cite{DK2012, KM2011, KZ2013, P2012, KP2017, KMP2022} for more details. One-dimensional excited random walks with non-nearest-neighbor steps were studied recently in \cite{DP2017, CHN2021, N2022}. Excited random walks on regular trees were considered in \cite{V2003} and \cite{BS2009}.
See also \cite{BW2003, Z2006, ABK2008, Hofstad2010, Menshikov2012, AHR2023} for results on excited random walks on $\Z^d$. For a literature review, we refer the reader to \cite{KZ2013}. Note that the behavior of OERW depends on its vertex range (i.e., the set of vertices visited up to the current time). This makes it closely related to the once-reinforced random walk (ORRW) introduced by Davis \cite{D1990}, whose behavior instead depends on the (unoriented) edge range. The ORRW has been studied extensively on trees \cite{DKL2002, C2006, KS2018, CHK2018, CKS2020}, on $\Z^d$ \cite{CT2025, EK2006}, and on non-amenable graphs \cite{CT2025}.

The OERW model on general trees has been considered in \cite{H2018}. The author in \cite{H2018} assumed that the bias parameters in the excited regime $\lambda_v=\gamma$ are the same for all $v\in V\setminus\{\rho\}$. However, under this homogeneous assumption, the walk does not exhibit a sharp phase transition between transience and recurrence with respect to parameter $\gamma$, even on a spherically symmetric tree (see Theorem 1 and Lemma 7 in \cite{H2018}). 
 In our present paper, we generalize the result in \cite{H2018} to the case when the parameters $(\lambda_v)_{v\in V}$ and $(\mu_v)_{v\in V}$  could take values locally depending on $v$ (see Theorem \ref{RT.thm} below).
 Furthermore, under a modified assumption for $(\lambda_v)$, we show that $\text{OERW}(\boldsymbol{\lambda})$ has a sharp phase transition on polynomially growing trees (see Theorem \ref{main.thm} below).

For each edge $e=\{v^{-1},v\}\in E$, let $|e|=|v|$ be the distance from $v$ to $\rho$, i.e., the number of edges in the shortest path connecting $v$ and $\rho$. A \textbf{cutset} in $\mathcal{T}$ is a minimal set $\pi$ of edges that separates the root from infinity. That is, for any infinite self-avoiding path $(v_i)_{i\ge 0}$ with $v_0=\rho$, there exists a unique $i$ such that $\{v_i, v_{i+1}\} \in \pi$.
Let $\Pi$ be the set of all cutsets in $\mathcal{T}$.
To measure polynomial growth, we use the \textbf{branching-ruin number}, introduced in \cite{CKS2020}:
\begin{equation}\label{brr2}
\text{br}_r(\mathcal{T}) = \sup\left\{\gamma > 0: \inf_{\pi\in \Pi}\sum_{e\in\pi}|e|^{-\gamma} > 0\right\}.
\end{equation}
Let $\partial \mathcal{T}$ stand for the \textbf{boundary} of $\mathcal{T}$, which is the set of all infinite self-avoiding paths starting from the root. We define the metric between any two infinite paths $\xi, \eta \in \partial \mathcal{T}$ whose last common edge is $e$ by
\begin{equation}\label{def.d}
    d(\xi, \eta) = 1/|e|.
\end{equation}
Note that ${\rm br}_r(\mathcal{T})$ is equal to the \textbf{Hausdorff dimension} of $\partial \mathcal{T}$ with respect to the metric $d$ (see Section 3.3 in \cite{CKS2020}).

In this paper, we establish an annealed phase transition between recurrence and transience for the once-excited random walk in random environment on the tree $\mathcal{T}$ with respect to the branching-ruin number $\text{br}_r(\mathcal{T})$. \vspace{5pt}

\begin{theorem}\label{main.thm} 
    Assume that 
    $$\lambda_v=1+\alpha_v \deg(v), \text{ and } \mu_{v}=1 \text{ for all } v\in V\setminus\{\rho\},$$
    where $(\alpha_v)_{v\in V\setminus\{\rho\}}$ are i.i.d. non-negative random variables with $m:=\mathbf{E}[1/(\alpha_v+1)]$.
    Then, under the annealed law, the process ${\rm OERW}(\boldsymbol{\lambda})$ is a.s. transient when ${\rm br}_r(\Tcal)> 2-m$ and it is a.s. recurrent when ${\rm br}_r(\Tcal)< 2-m$. 
\end{theorem}

We also prove a quenched phase transition for the generalized once-excited random walk. For each edge $e=\{v^{-1},v\}\in E$, denote by $\mathcal{P}_e$ or $\mathcal{P}_v$ the unique shortest path of edges connecting $v$ to $\rho$. For two edges $e_1=\{v_1^{-1},v_1\}$ and $e_2=\{v_2^{-1},v_2\}$, we write $e_1 \le e_2$ or $e_1\le v_2$ if $e_1\in \mathcal{P}_{e_2}$. We also write $e_1 < e_2$ or $e_1< v_2$ if $e_1\le e_2$ and $e_1\neq e_2$. Recall that for $v\in V\setminus\{\rho\}$, $v^{-1}$ is the parent vertex of $v$. For a vertex $v$ with $|v|\ge 2$, let $v^{-2}:= (v^{-1})^{-1}$ be the grandparent vertex of $v$.
For each edge $e\in E$, define 
\begin{align}\label{def.Psi}
\Psi(e)&=\prod_{g\le e} \psi(g),
\end{align}
in which
        \begin{align}\label{def.psi}
           \psi(u^{-1}, u) :=\begin{cases}
               1 & \text{ for } u\sim \rho, \\
                \displaystyle 1 - \left(1 - \frac{\varphi(u^{-2})}{\varphi(u)}\right)\cdot\frac{ \lambda_{u^{-1}} + (\deg(u^{-1})-2) \frac{\mu_{u^{-1}}}{\mu_{u^{-1}}+1}  }{ \lambda_{u^{-1}} + \deg(u^{-1}) - 1 } & \text { for } |u|\ge 2,
           \end{cases}  
\end{align}
with $\varphi(\rho)=0$ and, for each $x\in V$,  \begin{align}\label{def.phi}
    \varphi(x) := \sum_{e\in \mathcal{P}_x} R(e) \quad \text{with}\quad R(e):=\prod_{\{z^{-1}, z\} \in \mathcal{P}_{e} } \mu_{z^{-1}} \text{ for each $e\in E$.}
\end{align}
Let \[
RT(\Tcal, \boldsymbol{\lambda},\boldsymbol{\mu}) := \sup \left\{ \gamma > 0 : \inf_{\pi \in \Pi} \sum_{e \in \pi} (\Psi(e))^{\gamma} > 0 \right\}.
\]
 It is worth noting that $RT(\mathcal{T}, \boldsymbol{\lambda},\boldsymbol{\mu})$ is equal to the $\Psi$-Hausdorff dimension (i.e., the Hausdorff dimension with gauge function $\Psi$) of the boundary $\partial \mathcal{T}$ with respect to the metric defined in \eqref{def.d} (see Section 3.3 in \cite{CKS2020}). 
 
\begin{theorem}\label{RT.thm}
    Assume that $$
    \sup_{v\in V , |v|\ge 2}\frac{R(v^{-1},v)}{\varphi(v^{-1})}<\infty.$$
    Then,  under the quenched law $\mathbb{P}_{\boldsymbol{\lambda},\boldsymbol{\mu}}$, the ${\rm GOERW}(\boldsymbol{\lambda},\boldsymbol{\mu})$ is a.s. recurrent when $RT(\Tcal, \boldsymbol{\lambda},\boldsymbol{\mu}) < 1$, and a.s. transient when $RT(\Tcal, \boldsymbol{\lambda},\boldsymbol{\mu}) > 1$.
\end{theorem}
\vspace{5pt}

The remainder of the paper is organized as follows. In Section \ref{sec:constr}, we define a strong construction for the GOERW $\X$ on $\Tcal$ and introduce coupled processes $\X^{(v)}$ on $\mathcal{P}_v$, for $v\in V$, such that $\X^{(v)}$ coincides with the restriction of $\X$ to $\mathcal{P}_v$ until the last exit time from $\mathcal{P}_v$. Using this construction, we build in Section \ref{sec:quasi} a quasi-independent percolation model in which an edge $\{v^{-1},v\}$ is open if and only if the coupled process $\X^{(v)}$ hits $v$ before returning to $\rho$. We show that an edge $e$ is connected to the root by an open path with probability exactly equal to $\Psi(e)$. Using this fact, we prove that the process is transient when the percolation is supercritical and recurrent when it is subcritical. This connection to percolation allows us to establish criteria for transience and recurrence of the GOERW corresponding to the gauge function $\Psi$, as stated in Theorem \ref{RT.thm} (see the proof in Section \ref{sec:RT}). In the case of the OERW in random environment, we first show that $\Psi(e)$ concentrates around $|e|^{-2+m\pm\eps}$ with high probability. Using the Borel-Cantelli lemma, we deduce that there is a sequence of cutsets such that $\Psi(e)\le |e|^{-2+m+\eps}$ a.s. for each edge $e$ belonging to these cutsets.  This fact and Theorem \ref{RT.thm} yield the recurrence case of Theorem \ref{main.thm}. To prove the transience case of Theorem \ref{main.thm}, we use the quasi-independent percolation technique to show that, with positive probability, there is an infinite subtree $\widetilde{\Tcal}$ containing the root such that $\Psi(e)\ge \kappa^{-1} |e|^{-2+m-\eps}$, with $\kappa>0$, for each edge $e$ of $\widetilde{\Tcal}$ and $\text{br}_{r}(\widetilde{\Tcal})>\text{br}_{r}(\Tcal)-2\eps$. Combining this fact and Theorem \ref{RT.thm}, we deduce the transience case. We present the full proof of Theorem \ref{main.thm} in Section \ref{sec:main}.

Similarly to \cite{H2018}, our approach employs the ruin-probability technique introduced in \cite{CKS2020} for ORRW. For OERW, however, the proofs involve additional technical considerations compared with those in \cite{CKS2020}.
The unpublished paper \cite{H2018}, unfortunately, does not address the sharp phase transition for OERW, and certain steps in the proofs appear to contain gaps and inaccuracies (see Remark \ref{rem1} and Remark \ref{rem2}). Our results, however, are more general, extending to non-homogeneous bias parameters and random environments on any locally finite infinite tree.
The results in this paper thus complete the phase transition picture for once-excited random walks on trees and serve as the counterparts to the results for once-reinforced random walks established in \cite{CKS2020}.

\section{Strong construction of GOERW}\label{sec:constr}

   Let $\boldsymbol{\xi} = \big( \xi(v, u, j): \ v \in V, u \sim v, j \geq 0\big)$ be a collection of independent exponential random variables with rate 1. In this section, we use $\boldsymbol{\xi}$ to give a strong construction of the process $\mathbf{X}$ satisfying the quenched law $\mathbb{P}_{\boldsymbol{\lambda},\boldsymbol{\mu}}$ given by \eqref{P.quenched} and define coupled processes on each path $\mathcal{P}_v$ for $v\in V\setminus\{\rho\}$. This construction is inspired by Rubin's construction for the generalized Pólya urn (see Section 5 in \cite{D1990}).

    \subsection{Rubin's construction of $\X$}

Recall that $Z_n(v)= \sum_{j=0}^n \mathbf{1}_{\{X_j = v\}}$ is the number of visits to $v$ up to time $n$. Let $$C_n(v,u)=\sum_{j=1}^n \mathbf{1}_{\{X_{j-1}=v, X_j=u\}}$$ be the number of crossings from $v$ to $u$ up to time $n$. Let \begin{align}\label{r.def}
     r_{\boldsymbol{\lambda}}(v,u) & =\mathbf{1}_{\{u = v^{-1}\}} \prod_{\{w^{-1},w\}\in \mathcal{P}_{u}}\lambda_{w}^{-1} +\mathbf{1}_{\{u^{-1}=v\}}\prod_{\{w^{-1},w\}\in \mathcal{P}_{v}}\lambda_{w}^{-1} \quad\text{and}\\
    r_{\boldsymbol{\mu}}(v,u) &=\mathbf{1}_{\{u = v^{-1}\}} \prod_{\{w^{-1},w\}\in \mathcal{P}_{u}}\mu_{w}^{-1} +\mathbf{1}_{\{u^{-1}=v\}}\prod_{\{w^{-1},w\}\in \mathcal{P}_{v}}\mu_{w}^{-1}.
\end{align}
be the \textbf{rates} on the oriented edge $(v,u)$ associated with the excited and non-excited regimes at $v$, respectively.
Note that $r_{\boldsymbol{\lambda}}(v,v^{-1})=\lambda_{v}r_{\boldsymbol{\lambda}}(v,u)$ and $r_{\boldsymbol{\mu}}(v,v^{-1})=\mu_{v}r_{\boldsymbol{\mu}}(v,u)$ for each child $u$ of $v$.

Using the collection of exponential random variables $\boldsymbol{\xi}$, we construct the process $\X$ inductively as follows.

Set $X_0 = \rho$. On the event $\{X_n = v, Z_n(v)=1\}$, the next position is given by
$$X_{n+1}= \arg\min_{u\sim v} \frac{\xi(v,u,0)}{r_{\boldsymbol{\lambda}}(v,u)}.
$$
On this event, the process jumps from $v$ to its parent with probability $$\frac{r_{\boldsymbol{\lambda}}(v,v^{-1})}{\sum_{w\sim v}r_{\boldsymbol{\lambda}}(v,w)}= \frac{\lambda_{v}}{\lambda_v+\deg(v)-1}$$ or to a child $u$  with probability $$\frac{r_{\boldsymbol{\lambda}}(v,u)}{\sum_{w\sim v}r_{\boldsymbol{\lambda}}(v,w)}=\frac{1}{\lambda_v+\deg(v)-1}.$$ 

Let
$\tau_v^+:=\inf\{ n\ge 1 : Z_{n}(v)=2\}$ be the first return time to $v$.
On the event $\{X_n = v, Z_n(v)\ge 2\}$, the next position is given by
$$X_{n+1}= \arg\min_{u\sim v}\left\{ \sum_{k=C_{\tau_v^+}(v,u)}^{{C}_n(v,u)}\frac{\xi(v,u, k+1)}{r_{\boldsymbol{\mu}}(v,u)}\right\}.$$
Note that $$C_{\tau_v^+}(v,u)=\left\{ \begin{matrix}
    1 & \text{if $X_{\tau_v+1}=u$},\\
    0 & \text{otherwise},
\end{matrix}\right.$$
where $\tau_v:=\inf\{n\ge0 : X_n=v\}$ is the first hitting time to $v$.
By the memoryless property, we infer that, on the event $\{X_n = v, Z_n(v)\ge 2\}$, the process jumps from $v$ to its parent with probability ${\mu_{v}}/(\mu_v+\deg(v)-1)$ or to each of its children with probability ${1}/(\mu_v+\deg(v)-1)$.

Hence, the above-constructed process $\X=(X_n)_{n\ge0}$ is a ${\rm GOERW}(\boldsymbol{\lambda},\boldsymbol{\mu})$ with the quenched law $\P_{\boldsymbol{\lambda},\boldsymbol{\mu}}$ given by  \eqref{P.quenched}.  
\subsection{Restrictions and Extensions}

For a subset $B\subset V$, let $$\delta_n(B):=\inf\Big\{ k\ge 0: \sum_{j=0}^k\mathbf{1}_{\{X_j\in B\}}=n\Big\} \quad\text{and}\quad S_{B}:=\sup\{n\ge0 : \delta_n(B)<\infty\},$$
where we adopt the conventions $\inf \emptyset = \infty$ and $\sup\emptyset =-\infty$. 
Define $m_0 = 0$ and $m_{k+1} = \min\{ j > m_k : X_{\delta_j(B)} \neq X_{\delta_{m_k}(B)} \}$ for each $k\ge0$. Let $$K_B:=\sup\{k\ge 0: m_k\le S_{B}\}.$$
We call $(X_{\delta_{m_n}(B)})_{0\le n\le K_B}$ the \textbf{restriction} of the process $\X$ to $B$. This is a nearest-neighbor random walk which describes the movement of the process $\X$ within $B$, ignoring times when the process $\X$ is outside $B$. We call $K_B$ the \textbf{killing time} of the restriction to $B$. 

Fix a vertex $v \in V\setminus\{\rho\}$. Recall that $\mathcal{P}_v$ is the shortest path connecting $\rho$ and $v$. We now construct a process $\mathbf{X}^{(v)} = (X_n^{(v)})_{n \ge 0}$ on $\mathcal{P}_v$, which is coupled with $\X$ such that 
\begin{align}\label{coind}
    {X}^{(v)}_n=X_{\delta_{m_n}(\mathcal{P}_v)}\quad \text{for all $0\le n\le K_{\mathcal{P}_v}$,}
\end{align}
i.e., $\X^{(v)}$ coincides to the restriction of $\X$ to $\mathcal{P}_v$ up to the killing time $K_{\mathcal{P}_v}$.
This process is defined as follows.

Set $X_0^{(v)}=\rho$. For two vertices $u, w\in\mathcal{P}_v$ with $u\sim w$, let $Z^{(v)}_n(u)$, ${C}^{(v)}_n(u,w)$ be respectively the number of visits to vertex $u$ and the number of crossings from $u$ to $w$ by the process $\X^{(v)}$. For $u\in\mathcal{P}_v\setminus\{v\}$, denote by $(u_i)_{1\le i\le \deg(u)-1}$ the children of $u$. There exists a unique $j\in\{1,2,\cdots, \deg(u)-1\}$ such that $u_{j} \in\mathcal{P}_v$.
On the event $\{X_n^{(v)}=u, Z_n^{(v)}(u)=1\}$ with $u\neq v$, the next position is defined by
$$X_{n+1}^{(v)}=
\begin{cases}   u^{-1} & \text{ if } u^{-1}=\arg\min_{w\sim u} \frac{\xi(u,w,0)}{r_{\boldsymbol{\lambda}}(u,w)}, \\
u_j & \text{ if } u_j=\arg\min_{w\sim u} \frac{\xi(u,w,0)}{r_{\boldsymbol{\lambda}}(u,w)},\\   \arg\min_{w\in \{u^{-1},u_{j}\}} \frac{\xi(u,w,1)}{r_{\boldsymbol{\mu}}(u,w)} & \text{ otherwise}. \end{cases}
$$
Let $\tau_u^+(v)=\inf\{n\ge 1: Z_{n}^{(v)}(u)=2\}$ be the first return time of $\X^{(v)}$ to $u$.  On the event $\{X_n^{(v)}=u, Z_n^{(v)}(u)\ge 2\}$ with $u\neq v$, the next position is defined by
$$X_{n+1}^{(v)}= \arg\min_{w\in \{ u^{-1}, u_{j}\}}\left\{ \sum_{k={C}^{(v)}_{\tau_u^+(v)}(u,w)}^{{C}^{(v)}_n(u,w)}\frac{\xi(u,w, k+1)}{r_{\boldsymbol{\mu}}(u,w)}\right\}.$$
On the event $\{X_n^{(v)}=v\}$, set $X_{n+1}^{(v)}=v^{-1}$.

We call
$\mathbf{X}^{(v)} = (X_n^{(v)})_{n \ge 0}$ the \textbf{extension} of $\mathbf X$ on $\mathcal{P}_v$. One can easily verify that the property \eqref{coind} holds almost surely. 

\begin{remark}\label{rem1}\ 
\begin{itemize}
    \item The extension $\X^{(v)}$ is a nearest-neighbor random walk, but it is \underline{not} in general a ${\rm GOERW}$ on $\mathcal{P}_v$.
    \item The constructions of $\X$ and its extensions in Section 7 of \cite{H2018} do not satisfy \eqref{coind}. In fact, those extensions in \cite{H2018} used an extra collection of exponential random variables $(Z(u,v))_{u\in V, v\sim u}$ that does not appear in the construction of $\X$. This extra randomness does not guarantee that the extension on $\mathcal{P}_v$ coincides with the restriction of $\X$ to this path up to the killing time.
    \item In our paper, the property \eqref{coind} is essential for connecting the transience/recurrence of $\X$ to the quasi-independent percolation defined in Section \ref{sec:quasi} (see the proof of Lemma \ref{per2.lem} for details).
\end{itemize}
   
\end{remark}

\section{Quasi-independent percolation}
In this section, we define a quasi-independent percolation model in which an edge $e=\{v^{-1},v\}$ is open if and only if the extension process ${\X}^{(v)}$ reaches $v$ before returning to $\rho$. We prove that the probability that an edge $e$ belongs to the same percolation cluster as the root is exactly $\Psi(e)$, which is defined in \eqref{def.Psi}. This enables us to show that the process is transient when the percolation is supercritical and recurrent when it is subcritical. It is worth mentioning that quasi-independent percolation was first introduced by Lyons \cite{Lyons1989} and has since been applied to the study of once-reinforced random walks \cite{CKS2020}, random walks among random conductances \cite{CHK2019}, once-excited random walks with homogeneous biases \cite{H2018} and random walks in non-homogeneous random environment \cite{CNT2026}. Although computations for the homogeneous case were outlined in \cite{H2018}, some details appear to be incomplete and contain several inaccuracies. The proofs in this section not only complete the results in \cite{H2018} but also extend them to the general non-homogeneous case.

\subsection{Preliminary notations and results}
Let $\Tcal=(V,E)$ be a infinite locally finite tree. For each edge $e\in E$, we assign a Bernoulli random variable $\xi_e$ with parameter $p_e\in [0,1]$. The Bernoulli field $(\xi_e)_{e\in E}$ is not necessarily independent nor identically distributed. We say that $e$ is \textbf{open} if $\xi_e=1$, and \textbf{closed} otherwise. Assume that $(\xi_e)_{e\in E}$ is governed by a probability measure $\mathbf{Q}$. We call $\mathbf Q$ a \textbf{bond percolation} on $\Tcal$. 
After removing all closed edges from $E$, we obtain connected components comprising of open edges, which we call \textbf{clusters}. If two vertices $x$ and $y$ belong to the same cluster, we write $x\leftrightarrow y$. If the cluster containing $x$ has infinitely many vertices, we write $x\leftrightarrow \infty$. For two vertices $x,y$, we denote by $x\wedge y$ their nearest common ancestor.

A bond percolation $\mathbf{P}$ is \textbf{quasi-independent} if there exists a constant $M>0$ where 
$$\mathbf{Q}(\rho\leftrightarrow x, \rho\leftrightarrow y \ |\ \rho\leftrightarrow x\wedge y )\le M\cdot \mathbf{Q}(\rho\leftrightarrow x \ |\ \rho\leftrightarrow x\wedge y )\mathbf{Q}(\rho\leftrightarrow y \ |\ \rho\leftrightarrow x\wedge y )$$
for each $x, y\in V$. 

For each edge $e=\{e^-,e^+\} \in E$ with $|e^+|=|e^-|+1$, let 
\begin{align}\label{adt.c} \text{ $c(e) = 1$ for $|e| = 1$}\quad \text{and}\quad
    c(e) = \frac{\mathbf{Q}(\rho\leftrightarrow e^+)}{\mathbf{Q}(e \text{ is closed} \mid \rho\leftrightarrow e^-)}.
\end{align}
for $|e|>1$. We call $(c(e))_{e\in E}$ the \textbf{adapted conductances} of the percolation $\mathbf{Q}$. We will use the following result:

\begin{proposition}[Theorem 5.19 in \cite{LP2016}] \label{lyons1989}
    Let $\mathbf{Q}$ be a quasi-independent percolation process taking place on an infinite locally finite tree $\Tcal = (V,E)$.
    \begin{itemize}
        \item [(i)] If \(\inf_{\pi \in \Pi} \sum_{e \in \pi} \mathbf{Q}(\rho \leftrightarrow e) = 0 \text{ then } \mathbf{Q}(\rho \leftrightarrow \infty) = 0\);
        \item [(ii)] If there exists a non-zero flow $\theta$ such that $\sum_{e\in E} \frac{\theta(e)^2}{c(e)}<\infty$ then  $\mathbf{Q}(\rho \leftrightarrow \infty) > 0.$
    \end{itemize}
\end{proposition}

\subsection{Ruin percolation}\label{sec:quasi}
Let $\tau_v=\inf\{n\ge0: X_n=v\}$ and $\tau_v^+=\inf\{n>\tau_v: X_n=v\}$ be respectively the first hitting time of vertex $v$ and the first return time to vertex $v$. 
Define
$$
\mathcal{C} ( \rho)=\left\{\{v^{-1},v\} \in E \colon \tau_v  < \tau_{\rho}^+ \right\}. 
$$ 

Let $\tau_u(v)$ and $\tau_{u}^{+}(v)$ be respectively the first hitting times and the return time to vertex $u$ associated with $\X^{(v)}$.
Let 
$$\mathcal{C}_{\rm CP} ( \rho)=\left\{\{v^{-1},v\} \in E \colon \tau_v(v)  < \tau_{\rho}^+(v) \right\}. $$
We say an edge $e \in E$ is \textbf{open} if $e \in{\mathcal{C}}_{\mathrm{CP}} ( \rho)$, and \textbf{closed} otherwise. We define a correlated percolation by removing all closed edges, and we refer to this model as the \textbf{ruin percolation}.

For each edge $\{u^{-1},u\}\in\mathcal{P}_v$, let 
$$\mathcal{A}_{u}(v):=\left\{ \text{$ \X^{(v)} $ reaches $u$ before returning to $\rho$ since the first hitting time to $u^{-1}$}\right\}.$$

We now generalize Lemma 14 in \cite{H2018} to compute the exact value of the ruin probability $\Psi(e)$.

   \begin{lemma}\label{phical}
    For each edge $\{u^{-1},u\}\in \mathcal{P}_v$, we have 
$$\P_{\boldsymbol{\lambda}, \boldsymbol{\mu}}(\mathcal{A}_{u}(v))=\P_{\boldsymbol{\lambda}, \boldsymbol{\mu}}(\mathcal{A}_{u}(u))=\psi(u^{-1},u)$$
which does not depend on $v$. Furthermore, 
   \begin{align}
\P_{\boldsymbol{\lambda}, \boldsymbol{\mu}}(e\in \mathcal{C}_{\rm CP}(\rho))=\Psi(e)&=\prod_{g\le e} \psi(g).
\end{align}
    \end{lemma}

\begin{proof}
The event $\big\{\{v^{-1}, v\} \in \mathcal{C}_{\rm CP}\big\}=\{\tau_v(v) < \tau_\rho^+(v)\}$ means that the process $\mathbf{X}^{(v)}$ reaches $v$ before returning to $\rho$. This can be written as:
\[
\{\tau_v(v) < \tau_\rho^+(v)\} = \bigcap_{\{u^{-1},u\}\in\mathcal{P}_v} R_u(v)
\]
where 
$R_u(v)=\{\X^{(v)}  \text{ reaches  $u^{-1}$ and then reaches $u$ before returning to $\rho$}\}.$
It is clear that the event $R_u(v)$  conditional on $\bigcap_{\{w^{-1},w\}\in \mathcal{P}_{u^{-1}}}R_{w}(v)$ is equal to the event $\mathcal{A}_{u}(v)$. By the conditional probability formula, we obtain
\[
\mathbb{P}_{\boldsymbol{\lambda}, \boldsymbol{\mu}}(\{v^{-1}, v\} \in \mathcal{C}_{\rm CP})=\mathbb{P}_{\boldsymbol{\lambda}, \boldsymbol{\mu}}(\tau_v(v) < \tau_\rho^+(v)) = \prod_{\{u^{-1},u\}\in \mathcal{P}_v} \P_{\boldsymbol{\lambda}, \boldsymbol{\mu}}(\mathcal{A}_{u}(v)).
\]

We now compute $\P_{\boldsymbol{\lambda}, \boldsymbol{\mu}}(\mathcal{A}_{u}(v))$ for $\{u^{-1},u\}\in\mathcal{P}_v$. When $u\sim \rho$, it is clear that $\P_{\boldsymbol{\lambda}, \boldsymbol{\mu}}(\mathcal{A}_{u}(v))=\P_{\boldsymbol{\lambda}, \boldsymbol{\mu}}(\mathcal{A}_{u}(u))=1$. Assume from now {on} that $|u|\ge 2$. Recall that $\mathcal{A}_{u}(v)$ is the event that the process $\X^{(v)}$ reaches $u$ before returning to $\rho$ since the first hitting time to $u^{-1}$. It is clear that $\X^{(v)}$ has visited all vertices from $\rho$ to $u^{-1}$ by that time.
Let
\begin{align*}
    \mathcal{I}_0:=\Big\{ u^{-2}=\arg\min_{w\sim u^{-1}} \frac{\xi(u^{-1},w,0)}{r_{\boldsymbol{\lambda}}(u^{-1},w)}\Big\},\  
\mathcal{I}_1 := \Big\{u=\arg\min_{w\sim u^{-1}} \frac{\xi(u^{-1},w,0)}{r_{\boldsymbol{\lambda}}(u^{-1},w)}\Big\} \text{  and }
\mathcal{I}_2 := (\mathcal{I}_{0}\cup \mathcal{I}_{1})^c.
\end{align*}
Note that
\begin{equation}\label{prob.Ik}\begin{aligned}
    \P_{\boldsymbol{\lambda}, \boldsymbol{\mu}}(\mathcal{I}_0)&=\frac{\lambda_{u^{-1}}}{\lambda_{u^{-1}}+\deg({u^{-1}})-1},\\ \P_{\boldsymbol{\lambda}, \boldsymbol{\mu}}(\mathcal{I}_1)&=\frac{1}{\lambda_{u^{-1}}+\deg({u^{-1}})-1}\ \ \text{ and}\\
     \P_{\boldsymbol{\lambda}, \boldsymbol{\mu}}(\mathcal{I}_2)&=1-\frac{1}{\lambda_{u^{-1}}+\deg({u^{-1}})-1}-\frac{\lambda_{u^{-1}}}{\lambda_{u^{-1}}+\deg({u^{-1}})-1}=\frac{\deg({u^{-1}})-2}{\lambda_{u^{-1}}+\deg({u^{-1}})-1}. \end{aligned}\end{equation}

By the construction of the process $\X^{(v)}$, we have that
\begin{align}\label{P(A)}
    \P_{\boldsymbol{\lambda}, \boldsymbol{\mu}}(\mathcal{A}_{u}(v)) =& \sum_{k\in\{0,1,2\}}\P_{\boldsymbol{\lambda}, \boldsymbol{\mu}}(\mathcal{A}_{u}(v) \mid \mathcal{I}_k)\cdot\P_{\boldsymbol{\lambda}, \boldsymbol{\mu}}(\mathcal{I}_k).
\end{align}

Note that on $\mathcal{I}_0=\big\{ u^{-2}=\arg\min_{w\sim u^{-1}} \frac{\xi(u^{-1},w,0)}{r_{\boldsymbol{\lambda}}(u^{-1},w)}\big\}$, right after the first hitting time to $u^{-1}$, the process $\X^{(v)}$ moves from $u^{-1}$ to $u^{-2}$. At this time, each vertex $w\in \Pcal_{u}\setminus\{\rho,u\}$ has been visited and switched to non-excited mode. Therefore, the event
$\mathcal{A}_{u}(v)$ conditional on $\mathcal{I}_0$
is equivalent to the gambler's ruin problem on $\mathcal{P}_u$ {with bias $\mu_w$ at each state $w \in \Pcal_{u}\setminus\{\rho,u\}$}, where the gambler starts at state $u^{-2}$ and reaches $u$ before $\rho$.
Using the ruin probability formula in Section \ref{sec:gam}, we obtain that
\begin{align}\label{P(A|I0)ge}\P_{\boldsymbol{\lambda}, \boldsymbol{\mu}}\big(\mathcal{A}_{u}(v) \mid \mathcal{I}_0\big)=\frac{\varphi(u^{-2})}{\varphi(u)},\end{align}
where we recall that $${\varphi(x) = \sum_{\{y^{-1},y\}\in \mathcal{P}_x}\prod_{\{z^{-1},z\} \in \mathcal{P}_{y^{-1}}} \mu_{z}.}$$

 Conditional on $\mathcal{I}_1=\big\{u=\arg\min_{w\sim u^{-1}} \frac{\xi(u^{-1},w,0)}{r_{\boldsymbol{\lambda}}(u^{-1},w)}\big\}$, the process $\X^{(v)}$ moves from $u^{-1}$ to $u$ right after the first hitting to $u$, and the event $\mathcal{A}_{u}(v)$ occurs immediately. Hence
\begin{align}\label{P(A|I1)ge}\P_{\boldsymbol{\lambda}, \boldsymbol{\mu}}(\mathcal{A}_{u}(v) \mid \mathcal{I}_1)=1.\end{align}

Conditional on $\mathcal{I}_2=(\mathcal{I}_0\cup \mathcal{I}_1)^c$, the process $\X^{(v)}$ moves from $u^{-1}$ to its neighbors on $\mathcal{P}_u$ with bias $\mu_{u^{-1}}$ right after the first hitting to $u^{-1}$. {At this time, each vertex $w\in \Pcal_{u}\setminus\{\rho,u\}$ has been visited and switched to non-excited mode with bias $\mu_{w}$.
Define
\begin{align*}
    \mathcal{B}_1&=\{\text{$\X^{(v)}$ moves from $u^{-1}$ to $u^{-2}$ right after the first hitting time to $u^{-1}$}\},\\
    \mathcal{B}_2&=\{\text{$\X^{(v)}$ moves from $u^{-1}$ to $u$ right after the first hitting time to $u^{-1}$}\}.
\end{align*}
By the construction, we have
$$\P_{\boldsymbol{\lambda}, \boldsymbol{\mu}}(\mathcal{B}_1\mid \mathcal{I}_2)=\frac{\mu_{u^{-1}}}{\mu_{u^{-1}}+1} \text{ and } \P_{\boldsymbol{\lambda}, \boldsymbol{\mu}}(\mathcal{B}_2\mid \mathcal{I}_2)=\frac{1}{\mu_{u^{-1}}+1}.$$
We also notice that:
\begin{itemize}
    \item The conditional event $\mathcal{A}_{u}(v)$ given
$\mathcal{I}_2\cap \mathcal{B}_1$ is equivalent to the same gambler's ruin problem related to the conditional event $\mathcal{A}_{u}(v)$ given $\mathcal{I}_0$. Hence
$$\P_{\boldsymbol{\lambda}, \boldsymbol{\mu}}(\mathcal{A}_{u}(v)  \mid \mathcal{I}_2\cap \mathcal{B}_1)=\P_{\boldsymbol{\lambda}, \boldsymbol{\mu}}(\mathcal{A}_{u}(v) \mid \mathcal{I}_0)=\frac{\varphi(u^{-2})}{\varphi(u)}.$$
\item The conditional event $\mathcal{A}_{u}(v)$ given
$\mathcal{I}_2\cap \mathcal{B}_2$ is equivalent to the same gambler's ruin problem related to the conditional event $\mathcal{A}_{u}(v)$ given $\mathcal{I}_1$. Hence
$$\P_{\boldsymbol{\lambda}, \boldsymbol{\mu}}(\mathcal{A}_{u}(v) \mid \mathcal{I}_2\cap \mathcal{B}_2)=\P_{\boldsymbol{\lambda}, \boldsymbol{\mu}}(\mathcal{A}_{u}(v) \mid \mathcal{I}_1)=1.$$
\end{itemize}
}
By applying the law of total probability, we get:
$$\P_{\boldsymbol{\lambda}, \boldsymbol{\mu}}(\mathcal{A}_{u}(v) \mid \mathcal{I}_2)= \P_{\boldsymbol{\lambda}, \boldsymbol{\mu}}(\mathcal{A}_{u}(v) \mid \mathcal{I}_2 \cap \mathcal{B}_1)\cdot \P_{\boldsymbol{\lambda}, \boldsymbol{\mu}}(\mathcal{B}_1\mid\mathcal{I}_2)+ \P_{\boldsymbol{\lambda}, \boldsymbol{\mu}}(\mathcal{A}_{u}(v) \mid \mathcal{I}_2 \cap \mathcal{B}_2) \cdot \P_{\boldsymbol{\lambda}, \boldsymbol{\mu}}(\mathcal{B}_2\mid\mathcal{I}_2).$$
Therefore
\begin{align}\label{P(A|I2)ge}\P_{\boldsymbol{\lambda}, \boldsymbol{\mu}}(\mathcal{A}_{u}(v) \mid \mathcal{I}_2)=\frac{\varphi(u^{-2})}{\varphi(u)}.\frac{\mu_{u^{-1}}}{\mu_{u^{-1}}+1}+\frac{1}{\mu_{u^{-1}}+1}.\end{align}

From \eqref{prob.Ik}, \eqref{P(A)}, \eqref{P(A|I0)ge}, \eqref{P(A|I1)ge} and \eqref{P(A|I2)ge}, we obtain
   \begin{align*}\P_{\boldsymbol{\lambda}, \boldsymbol{\mu}}(\mathcal{A}_u(v))&= \frac{\varphi(u^{-2})}{\varphi(u)}.\frac{\lambda_{u^{-1}}}{\lambda_{u^{-1}}+\deg({u^{-1}})-1} +\frac{1}{\lambda_{u^{-1}}+\deg({u^{-1}})-1}.\\
   &+\bigg(\frac{\varphi(u^{-2})}{\varphi(u)}.\frac{\mu_{u^{-1}}}{\mu_{u^{-1}}+1}+\frac{1}{\mu_{u^{-1}}+1}\bigg).\frac{\deg({u^{-1}})-2}{\lambda_{u^{-1}}+\deg({u^{-1}})-1}\\
    & = 1 - \frac{ \left(1 - \frac{\varphi(u^{-2})}{\varphi(u)}\right) \left( \lambda_{u^{-1}} + (\deg(u^{-1})-2) \frac{\mu_{u^{-1}}}{\mu_{u^{-1}}+1} \right) }{ \lambda_{u^{-1}} + \deg(u^{-1}) - 1 }=\psi(u^{-1},u).\end{align*}
\end{proof}

The next lemma asserts that the process 
$\X$ is a.s. transient when the percolation cluster 
$\mathcal{C}_{\rm CP}(\rho)$ is infinite with positive probability, and the process is a.s. recurrent when this cluster is a.s. finite. A similar result was proved for once-reinforced random walks in Lemma 7.1 of \cite{CKS2020}. For the sake of completeness, we include a proof which completes details that were omitted in \cite{CKS2020}.\\

\begin{lemma} \label{per2.lem} We have
\begin{align}\label{event.equiv}
    {\P_{\boldsymbol{\lambda}, \boldsymbol{\mu}}}(\tau^+_{\rho} = \infty) = {\P_{\boldsymbol{\lambda}, \boldsymbol{\mu}}} (|\mathcal{C}(\rho)| = \infty) = {\P_{\boldsymbol{\lambda}, \boldsymbol{\mu}}}(|\mathcal{C}_{\rm CP}(\rho)| = \infty).
\end{align} 
Consequentially, under probability measure $\P_{\boldsymbol{\lambda}, \boldsymbol{\mu}}$, the process $\X$ \begin{itemize}
    \item is a.s. recurrent if $ {\P_{\boldsymbol{\lambda}, \boldsymbol{\mu}}}(|\mathcal{C}_{\rm CP}(\rho)| < \infty)=1$, 
\item is a.s. transient if ${\P_{\boldsymbol{\lambda}, \boldsymbol{\mu}}}(|\mathcal{C}_{\rm CP}(\rho)| = \infty)>0$.
\end{itemize}
\end{lemma}

    \begin{proof}
        (a) We first show that a.s. 
        $$\{\tau_{\rho}^+=\infty\}=\{|\mathcal{C}(\rho)|=\infty\}.$$

Suppose that $\tau^+_\rho = \infty$.  
Then the process $\X$ never returns to $\rho$ after the starting time. For each vertex $v$ visited by the process, the  edge $\{v^{-1}, v\}$ belongs to the cluster $\mathcal{C}(\rho)$. Since the walk continues indefinitely without returning, it must visit infinitely many distinct vertices. Therefore, the cluster $\mathcal{C}(\rho)$ contains infinitely many edges, and hence
$
\{\tau^+_\rho = \infty\} \subseteq \{|\mathcal{C}(\rho)| = \infty\}.
$

Conversely, if $\tau^+_\rho < \infty$, then only the vertices visited by $\X$ before that first return contribute edges to $\mathcal{C}(\rho)$.  
Because the tree is locally finite, the number of vertices visited before a finite time is finite, so $|\mathcal{C}(\rho)|$ must also be finite.  
Hence,
$
\{|\mathcal{C}(\rho)| = \infty\} \subseteq \{\tau^+_\rho = \infty\}.
$
Therefore, we obtain
$$ \{|\mathcal{C}(\rho)| = \infty\} =  \{\tau^+_\rho = \infty\}. $$

(b) To complete the proof for \eqref{event.equiv}, we show that a.s.
$$\{|\mathcal{C}(\rho)|=\infty\} = \{|\mathcal{C}_{\rm CP}(\rho)|=\infty\}.$$

Assume first that $|\mathcal{C}_{\rm CP}(\rho)| = \infty$. 
Then for every $n \ge 0$, there exists an edge $e = \{v^{-1}, v\}$ with $|v| = n$ such that 
$\tau_v(v) < \tau_\rho^+(v).$ Recall that $K_{\mathcal{P}_v}$ is the kill time of the restriction of $\X$ to $\mathcal{P}_v$. By the construction of the extension $\X^{(v)}$, the process $\X^{(v)}$ coincides with the restriction of $\X$ to $\mathcal{P}_v$ up to the kill time  $K_{\mathcal{P}_v}$. We distinguish the following two cases:
\begin{itemize}
    \item On the event $\{K_{\mathcal{P}_v}=\infty\}$ the process $\X^{(v)}$ coincides with the restriction of $\X$ to $\mathcal{P}_v$ for all time. Hence, it is clear that
$\{K_{\mathcal{P}_v}=\infty\}\cap\{\tau_v(v) < \tau_\rho^+(v)\}\subset\{
\tau_v < \tau_\rho^+\}.
$
\item On the event $\{K_{\mathcal{P}_v}<\infty\}$: if $\tau_v(v) < \tau_\rho^+(v)\le K_{\mathcal{P}_v}$ then it is clear that $\tau_v < \tau_\rho^+$;  if $\tau_v(v) < \tau_\rho^+(v)$ and $K_{\mathcal{P}_v}<\tau_\rho^+(v)$ then after starting at $\rho$, the process $\X$ eventually exits $\mathcal{P}_v$ at a vertex $u\in \mathcal{P}_{v}\setminus\{\rho\}$ and never returns back to $\mathcal{P}_v$ to visit $\rho$ for the second time. Hence, $\{K_{\mathcal{P}_v}<\infty\}\cap\{\tau_v(v) < \tau_\rho^+(v)\}\subset \{\tau_{\rho}^+=\infty\}\cup \{\tau_v < \tau_\rho^+\}.$
\end{itemize}
As one of the above two cases occurs  with infinite many $v$, we obtain that  
$$
\{|\mathcal{C}_{\rm CP}(\rho)| = \infty\}\subset  \{|\mathcal{C}(\rho)| = \infty\}.
$$

Conversely, suppose that $|\mathcal{C}(\rho)| = \infty$. 
Then for every $n \ge 0$, there exists an edge $e = \{v^{-1}, v\}$ with $|v| = n$ such that 
$\tau_v < \tau_\rho^+.$   We distinguish the following two cases:
   \begin{itemize}
       \item On the event $\{K_{\mathcal{P}_v}=\infty\}$, 
    the restriction of $\X$ to $\mathcal{P}_v$ and the extension $\X^{(v)}$ coincide for all time and thus
    $\{K_{\mathcal{P}_v}=\infty\}\cap
    \{\tau_v < \tau_\rho^+\} \subset \{\tau_v(v) < \tau_\rho^+(v)\}.
    $  
 \item   On the event $\{K_{\mathcal{P}_v}<\infty\}$: if $\tau_v < \tau_\rho^+$ and $\tau_{\rho}^+(v)\le K_{\mathcal{P}_v}$ then it is clear that $\tau_v(v) < \tau_\rho^+(v)$; if $\tau_v < \tau_\rho^+$ and $K_{\mathcal{P}_v}<\tau_{\rho}^+(v)$ then after starting from $\rho$ the process $\X$ eventually exits $\mathcal P_v$ at some vertex $u\in\mathcal P_v\setminus\{\rho\}$ and never returns to $\mathcal P_v$ to visit $\rho$ for the second time. In the later scenario, since $\tau_v<\tau_\rho^+=\infty$, the vertex $v$ is visited by this permanent exit, yielding that $\tau_v(v)\le K_{\mathcal{P}_v}< \tau_\rho^+(v)$. Hence,
    $\{K_{\mathcal{P}_v}<\infty\}\cap  \{\tau_v < \tau_\rho^+ \}  \subset \{\tau_v(v)<\tau_\rho^+(v)\}.$
       \end{itemize}
    As one of the above two cases occurs with infinite many $v$, we obtain that
    $$\{|\mathcal{C}(\rho)| = \infty\}\subset  \{|\mathcal{C}_{\rm CP}(\rho)| = \infty\}.$$

    (c) Following exactly the same argument in the proof Proposition 11.1 in \cite{CKS2020}, one can deduce that $\X$ is transient if and only if it visits $\rho$ infinitely often, and obtain the 0-1 law for transience and recurrence, i.e.
    $$\P_{\boldsymbol{\lambda}, \boldsymbol{\mu}}(\X \text{ is recurrent})=1-\P_{\boldsymbol{\lambda}, \boldsymbol{\mu}}(\X \text{ is transient})\in \{0,1\}.$$
Combine this with \eqref{event.equiv}, we obtain the conclusion of the lemma.
    \end{proof}

Recall that $\varphi(\rho)=0$ and, for each $x\in V$,  \begin{align*}
    \varphi(x) := \sum_{e\in \mathcal{P}_x} R(e) \quad \text{with}\quad R(e):=\prod_{\{z^{-1}, z\} \in \mathcal{P}_{e} } \mu_{z^{-1}} \text{ for each $e\in E$.}
    \end{align*}
\begin{lemma} \label{per1.lem}
    Assume that 
    $$\sup_{v\in V, |v|\ge 2}\frac{R(v^{-1},v)}{\varphi(v^{-1})}<\infty.$$ 
 Then under the quenched law $\P_{\boldsymbol{\lambda}, \boldsymbol{\mu}}$, the ruin percolation is quasi-independent.
\end{lemma}

\begin{proof}
        For an edge $g\in E$, denote by $g^+$ and $g^-$ the two endpoints of $g$ where $g^+$ is a child of $g^-$.
Fix two edges $e_1, e_2\in E$. If $e_1^+\wedge e_2^+ = \rho$, then the extensions $\X^{(e_1^+)}$ and $\X^{(e_2^+)}$ are independent, as they are defined by two disjoint collections of exponential random variables. In this case, we must have $$\mathbb{P}_{\boldsymbol{\lambda}, \boldsymbol{\mu}}(e_1,e_2\in \mathcal{C}_{\rm CP}(\rho))
  =\mathbb{P}_{\boldsymbol{\lambda}, \boldsymbol{\mu}}(e_1\in \mathcal{C}_{\rm CP}(\rho)\mathbb{P}_{\boldsymbol{\lambda}, \boldsymbol{\mu}}(e_2\in \mathcal{C}_{\rm CP}(\rho)).$$
  
  Assume from now that $e_1^+\wedge e_2^+ \neq \rho$. Let $e=\{e^-, e^+\}$ be the last common edge of $\mathcal{P}_{e_1}$ and $\mathcal{P}_{e_2}$. For $j\in\{1,2\}$, let $v_j$ be the child of $e^+$ lying on the path $\mathcal{P}_{e_j}$ (see Figure \ref{fig1}). Note that the extensions $\X^{(e_1^+)}$ and $\X^{(e_2^+)}$ are dependent as they are constructed with the same exponential random variables on the path connecting $\rho$ and $e^+$.
\begin{figure}[h]
\centering
\begin{tikzpicture}[
    every node/.style = {circle, fill=black, inner sep=0pt, minimum size=5pt},
    level distance = 1.5cm,
    sibling distance = 1.5cm,
    edge from parent/.style = {draw=black},
    grow = right       
  ]
\node[label=above:$\rho$] {}                      
  child { node {}                                 
    child { node[label=above:$e^-$] {}
      child { node[label=above:$e^+$] (lca) {}    
        child { node[label=above:$v_1$] (v1) {}
          child { node {} 
            child { node[label=above:$e_1^-$] {}
              child { node[label=above:$e_1^+$] {} }
            }
          }
        }
        child { node[label=above:$v_2$] (v2) {}
          child { node {}
            child { node[label=above:$e_2^-$] {}
              child { node[label=above:$e_2^+$] {} }
            }
          }
        }
      }
    }
  };
\end{tikzpicture}

\caption{The two shortest paths connecting $\rho$ with $e_1=\{e_1^-, e_1^+\}$ and $e_2=\{e_2^-, e_2^+\}$ with the last common edge $e=\{e^-,e^+\}$.} \label{fig1}
\end{figure}
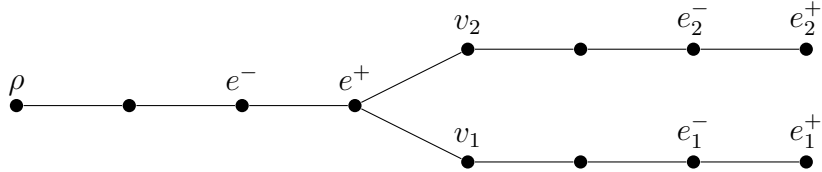

 Recall that $\tau_{e^+}(e^+)$ and $\tau_{e^+}^+(e^+)$ are the first hitting time and the first return time of $\X^{(e^+)}$ to $e^+$ respectively. Let $$\sigma_e=\{n\ge \tau_{e^+}^+(e^+)+1: X_{n}^{(e^+)}=\rho\}$$
 be the first time after $\tau_{e^+}^+(e^+)$ the process $\X^{(e^+)}$ returns to $\rho$. Let
\begin{align*}
N(e)&=\sum_{n=\tau_{e^+}^+(e^+)}^{\sigma_e}\mathbf{1}_{\big\{X^{(e^+)}_{n}=e^+,X^{(e^+)}_{n+1}=e^-\big\}}
\end{align*}
be the number of times the process $\X^{(e^+)}$ 
traverses from $e^{+}$ to $e^{-}$ since the first return time to $e^+$ before returning to the root. At time $\tau_{e^+}^+(e^+)$, all vertices in $\mathcal{P}_e$ have been visited. Therefore, from this time, the process $\X^{(e^+)}$ moves from a vertex $u\in \mathcal{P}_{e^-}$ to its neighbors with bias $\mu_u$. Note that a.s. $X^{(e^+)}_{\tau_{e^+}^+(e^+)+1}=e^-$ since the process $\X^{(e^+)}$ reflects from $e^+$ to $e^-$ with probability 1. 
We split the movements of the process $\X^{(e^+)}$ during $[\tau_{e^+}^+(e^+), \infty)$ into excursions between two consecutive visits to $e^-$. In each such excursion, after staying at $e^-$, the process reaches $\rho$ before $e^+$ with success probability $1- \frac{\varphi(e^{-})}{\varphi(e^{+})}$ (by applying the gambler ruin problem in Section \ref{gambler.ruin}). As these excursions are independent, $N(e)$ is a geometric random variable with parameter $1- \frac{\varphi(e^{-})}{\varphi(e^{+})}$.

Let
\begin{align*}
G(e)&= \sum_{n=1}^{N(e)} \frac{\xi(e^+,e^-,n+1)}{r_{\boldsymbol{\mu}}(e^+,e^-)}=R(e)\sum_{n=1}^{N(e)} \xi(e^+,e^-,n+1),
\end{align*}
which represents the total time consumed by the exponential clocks attached to the oriented edge $(e^+, e^-)$ since the first return time to $e^+$ before $\X^{(e^+)}$ or $\X^{(e_j^+)}$ goes back to the root. As $G(e)$ is a sum of $N(e)$ i.i.d. exponential random variables with rate $r_{\boldsymbol{\mu}}(e^{+}, e^{-})=1/R(e)$, it is exponentially distributed with rate $$p_e:=\Big(1-\frac{\varphi(e^{-})}{\varphi(e^{+})}\Big) \frac{1}{R(e)}=\frac{1}{\varphi(e^{+})}.$$
 
Let $\X^{[e^+,e_j^+]}$ be the extension of $\X$ on the shortest path connecting $e^+$ and $e_j^+$ which starts from $e^+$. Let
\begin{align*}
N^*(e_j) & = \sum_{n=\tau_{e^+}^+(e^+)}^{\tau_{e_j^+}(e_j^+)}\mathbf{1}_{\big\{X^{[e^+,e_j^+]}_{n}=e^+,X^{[e^+,e_j^+]}_{n+1}=v_j\big\}},\end{align*}
which represents the number of times the process $\X^{[e^+, e_j^+]}$ 
traverses from $e^{+}$ to $v_j$ since the first return time to $e^+$
before
the first hitting time to the vertex $e_j^+$. Let
\begin{align*}
G^*(e_j) & = \sum_{n=1}^{N^*(e_j)} \frac{\xi(e^+,v_j,n)}{r_{\boldsymbol{\mu}}(e^+,v_j)}.
\end{align*}
which is the time consumed by the exponential clocks attached to the oriented edge $(e^+, v_j)$
since the first return time to $e^+$ before $\mathbf{X}^{(e_j)}$, or $\mathbf{X}^{[e^+, e_j^+]}$ hits $e_j^+$ for the first time. Note that 
$G^*(e_1)$ and $G^*(e_2)$ are independent random variables as they are defined by two disjoint collections of exponential random variables. Denote by $f_1$ and $f_2$ the probability density functions of $G^*(e_1)$ and $G^*(e_2)$ respectively. 

To prove the quasi-independence, we compute the probabilities of the events $\{e_1\in \mathcal{C}_{\rm CP}(\rho)\}$, $\{e_2\in \mathcal{C}_{\rm CP}(\rho)\}$ and $\{e_1, e_2\in \mathcal{C}_{\rm CP}(\rho)\}$ conditional on $\{e\in \mathcal{C}_{\rm CP}(\rho)\}$ using the above-mentioned random variables $G(e)$, $G^*(e_1)$ and $G^*(e_2)$. For the sake of clarity, we separate the remainder of the proof into four steps.

\textbf{Step 1.} In this step, we show that for $j\in \{1,2\}$,
\begin{align}\label{condp.ej}
     & \mathbb{P}_{\boldsymbol{\lambda}, \boldsymbol{\mu}}(e_j \in \mathcal{C}_{\rm CP}(\rho) \mid e \in \mathcal{C}_{\rm CP}(\rho))
   \\
\nonumber   & =\left( 1 - \frac{ \left(1 - \frac{\varphi(e^-)}{\varphi(e^+)}\right) \left( \lambda_{e^+} + (\deg(e^+)-2) \frac{\mu_{e^+}}{\mu_{e^+}+1} \right) }{ \lambda_{e^+} + \deg(e^+) - 1 } \right) \mathbb{P}_{\boldsymbol{\lambda}, \boldsymbol{\mu}}(G(e) > G^*(e_j)) \text{ and}\\
& \label{condp.e1e2}
  \mathbb{P}_{\boldsymbol{\lambda}, \boldsymbol{\mu}}(e_1, e_2 \in \mathcal{C}_{\rm CP}(\rho) \mid e \in \mathcal{C}_{\rm CP}(\rho ))\\
  \nonumber  & = \left( 1 - \frac{ \left(1 - \frac{\varphi(e^-)}{\varphi(e^+)}\right) \left( \lambda_{e^+} + \frac{2\mu_{e^+}}{1+\mu_{e^+}} + \frac{(\deg(e^+)-3)\mu_{e^+}}{2+\mu_{e^+}}  \right) }{ \lambda_{e^+} + \deg(e^+) - 1 } \right) \mathbb{P}_{\boldsymbol{\lambda}, \boldsymbol{\mu}}\big(G(e) > \max\{G^*(e_1), G^*(e_2)\}\big). 
\end{align}
Let 
\begin{align*}
    \mathcal{I}_0&=
  \Big\{ e^{-}=\arg\min_{u\sim e^+} \frac{\xi(e^+,u,0)}{r_{\boldsymbol{\lambda}}(e^+,u)}\Big\}, 
\mathcal{I}_1 = \Big\{ v_1 =\arg\min_{u\sim e^+} \frac{\xi(e^+,u,0)}{r_{\boldsymbol{\lambda}}(e^+,u)}\Big\},  \mathcal{I}_2 = \Big\{ v_2 =\arg\min_{u\sim e^+} \frac{\xi(e^+,u,0)}{r_{\boldsymbol{\lambda}}(e^+,u)}\Big\},\\
\mathcal{I}_3 &= (\mathcal{I}_{0}\cup \mathcal{I}_{1}\cup \mathcal{I}_{2})^c.
\end{align*}
Note that the events $\mathcal
I_k$ with $k\in\{0,1,2,3\}$ are independent of $\{e\in \mathcal{C}_{\rm CP}(\rho)\}$.
By the law of total probability, we have
\begin{equation}\label{tplaw}
    \begin{aligned}
   & \mathbb{P}_{\boldsymbol{\lambda}, \boldsymbol{\mu}}(e_1 \in \mathcal{C}_{\rm CP}(\rho) \mid e \in \mathcal{C}_{\rm CP}(\rho)) =   \mathbb{P}_{\boldsymbol{\lambda}, \boldsymbol{\mu}}(e_1 \in \mathcal{C}_{\rm CP}(\rho) \mid \mathcal{I}_0,  e \in \mathcal{C}_{\rm CP}(\rho))\cdot \P(\mathcal{I}_0 )\\
  & \hspace{100pt} + \mathbb{P}_{\boldsymbol{\lambda}, \boldsymbol{\mu}}(e_1 \in \mathcal{C}_{\rm CP}(\rho) \mid \mathcal{I}_1,  e \in \mathcal{C}_{\rm CP}(\rho))\cdot \P(\mathcal{I}_1  )\\
 & \hspace{100pt} + \mathbb{P}_{\boldsymbol{\lambda}, \boldsymbol{\mu}}(e_1 \in \mathcal{C}_{\rm CP}(\rho) \mid \mathcal{I}_2\cup \mathcal{I}_3,  e \in \mathcal{C}_{\rm CP}(\rho))\cdot \P(\mathcal{I}_2\cup \mathcal{I}_3 ).
\end{aligned}
 \end{equation}
Notice that \begin{align}\label{prob.I}
    \mathbb{P}(\mathcal{I}_0)=\frac{\lambda_{e^+}}{\lambda_{e^+}+\deg(e^+)-1}, \P(\mathcal{I}_1)=\frac{1}{\lambda_{e^+}+\deg(e^+)-1}, \P(\mathcal{I}_2\cup \mathcal{I}_3)=\frac{\deg(e^+)-2}{\lambda_{e^+}+\deg(e^+)-1}.
\end{align}

On $\mathcal{I}_0\cap \{e\in \mathcal{C}_{\rm CP}(\rho)\}$, the process $\X^{(e_1^+)}$ jumps to $e^-$ right after the first hitting time to $e^+$. The event that $e_1$ belongs to the cluster $\mathcal{C}_{\rm CP}(\rho)$ conditional on $\{e\in\mathcal{C}_{\rm CP}(\rho)\}\cap\mathcal{I}_0$ is equivalent to the intersection of the two following independent events:
\begin{itemize}
    \item The process $\X^{(e^+_1)}$ after the first jump from $e^+$ to $e^-$ will hit $e^+$ again before returning to $\rho$. By the gambler ruin's problem, this event happens with probability $\frac{\varphi(e^-)}{\varphi(e^+)}$.
    \item $G(e)>G^*(e_1)$. This event allows the process $\X^{(e_1^+)}$ after the first return to $e^+$ will hit $e_1^+$ before going to $\rho$.  
\end{itemize}
Hence,
        \begin{align}\label{case.I0}
\mathbb{P}_{\boldsymbol{\lambda}, \boldsymbol{\mu}}(e_1 \in \mathcal{C}_{\rm CP}(\rho) \mid e \in \mathcal{C}_{\rm CP}(\rho), \mathcal{I}_0) = \frac{\varphi(e^-)}{\varphi(e^+)} \mathbb{P}_{\boldsymbol{\lambda}, \boldsymbol{\mu}}\big(G(e) > G^*(e_1)\big) = \frac{\varphi(e^-)}{\varphi(e^+)} \int_0^{+\infty} e^{-p_e x} f_1(x)\rmd x. \end{align}

On $\mathcal{I}_1\cap \{e\in \mathcal{C}_{\rm CP}(\rho)\}$, the process $\X^{(e_1^+)}$ jumps to $v_1$ right after the first hitting time to $e^+$. Before returning to $\rho$, the process $\X^{(e_1^+)}$ must return to $e^+$. The event that $e_1\in \mathcal{C}_{\rm CP}(\rho)$ conditional on $\{e\in\mathcal{C}_{\rm CP}(\rho)\}\cap\mathcal{I}_1$ is equivalent to the  $G(e)>G^*(e_1)$, which allows the process $\X^{(e_1^+)}$ after the first return to $e^+$ will reach $e_j^+$ before going to $\rho$.  
Hence,
        \begin{align}\label{case.I1}
\mathbb{P}_{\boldsymbol{\lambda}, \boldsymbol{\mu}}(e_1 \in \mathcal{C}_{\rm CP}(\rho) \mid e \in \mathcal{C}_{\rm CP}(\rho), \mathcal{I}_1) = \mathbb{P}_{\boldsymbol{\lambda}, \boldsymbol{\mu}}(G(e) > G^*(e_1). \end{align}

On $(\mathcal{I}_2\cup \mathcal{I}_3)\cap \{e\in \mathcal{C}_{CP}(\rho)\}$, after the first hitting time to $e^+$, the process $\X^{(e_1^+)}$ jumps to either $e^-$ with probability $\frac{\mu_{e^+}}{1+\mu_{e^+}}$ or $v_1$ with probability $\frac{1}{1+\mu_{e^+}}$. The event that $e_1$ belongs to the cluster $\mathcal{C}_{\rm CP}(\rho)$ conditional on $\{e\in\mathcal{C}_{\rm CP}(\rho)\}\cap(\mathcal{I}_2\cup \mathcal{I}_3)$ is equivalent to the intersection of the two following independent events:
\begin{itemize}
    \item The process $\X^{(e^+_1)}$ after the first jump from $e^+$ will hit $e^+$ again before returning to $\rho$. This event happens with probability $\frac{\mu_{e^+}}{1+\mu_{e^+}}\frac{\varphi(e^-)}{\varphi(e^+)}+\frac{1}{1+\mu_{e^+}}$.
    \item $G(e)>G^*(e_1)$. This event allows the process $\X^{(e_1^+)}$ after the first return to $e^+$ will hit $e_1^+$ before going to $\rho$.  
\end{itemize}
Hence,
        \begin{align}\label{case.I2I3}
\mathbb{P}_{\boldsymbol{\lambda}, \boldsymbol{\mu}}(e_1 \in \mathcal{C}_{CP}(\rho) \mid e \in \mathcal{C}_{\rm CP}(\rho), \mathcal{I}_2\cup \mathcal{I}_3) = \left( \frac{\mu_{e^+}}{1+\mu_{e^+}}\frac{\varphi(e^-)}{\varphi(e^+)}+\frac{1}{1+\mu_{e^+}} \right) \mathbb{P}_{\boldsymbol{\lambda}, \boldsymbol{\mu}}(G(e) > G^*(e_1)).
\end{align}

Combining \eqref{prob.I}-\eqref{case.I0}-\eqref{case.I1}-\eqref{case.I2I3} with \eqref{tplaw},
we obtain
\begin{align*}& \mathbb{P}_{\boldsymbol{\lambda}, \boldsymbol{\mu}}(e_1 \in \mathcal{C}_{CP}(\rho) \mid e \in \mathcal{C}_{\rm CP}(\rho)) \\
& = \left(  \frac{\varphi(e^-)}{\varphi(e^+)} \frac{\lambda_{e^+}}{\lambda_{e^+}+\deg(e^+)-1}  + \frac{1}{\lambda_{e^+}+\deg(e^+)-1}  \right.\\
& \left.+ \Big(\frac{\mu_{e^+}}{1+\mu_{e^+}}\frac{\varphi(e^-)}{\varphi(e^+)}+\frac{1}{1+\mu_{e^+}}\Big) \frac{\deg(e^+)-2}{\lambda_{e^+}+\deg(e^+)-1} \right)  \cdot \mathbb{P}_{\boldsymbol{\lambda}, \boldsymbol{\mu}}(G(e) > G^*(e_1)) \\
& = \left( 1 -\left(1 - \frac{\varphi(e^-)}{\varphi(e^+)}\right)\cdot \frac{  \lambda_{e^+} + (\deg(e^+)-2) \frac{\mu_{e^+}}{\mu_{e^+}+1}  }{ \lambda_{e^+} + \deg(e^+) - 1 }\right) \cdot \mathbb{P}_{\boldsymbol{\lambda}, \boldsymbol{\mu}}(G(e) > G^*(e_1)).
\end{align*}
By symmetry, we also obtain the formula for $\mathbb{P}_{\boldsymbol{\lambda}, \boldsymbol{\mu}}(e_1 \in \mathcal{C}_{CP}(\rho) \mid e \in \mathcal{C}_{\rm CP}(\rho))$ . Hence \eqref{condp.ej} is verified. Using the same argument, we have that  \begin{align*}
\mathbb{P}_{\boldsymbol{\lambda}, \boldsymbol{\mu}}(e_1, e_2 \in \mathcal{C}_{\rm CP}(\rho) \mid e \in \mathcal{C}_{\rm CP}(\rho ), \mathcal{I}_0)&= \frac{\varphi(e^-)}{\varphi(e^+)}\cdot \mathbb{P}_{\boldsymbol{\lambda}, \boldsymbol{\mu}}\big(G(e) > \max\{G^*(e_1), G^*(e_2)\}\big),\\
\mathbb{P}_{\boldsymbol{\lambda}, \boldsymbol{\mu}}(e_1, e_2 \in \mathcal{C}_{\rm CP }(\rho) \mid e \in \mathcal{C}_{\rm CP}(\rho ) , \mathcal{I}_1\cup \mathcal{I}_2)& =\left( \frac{\mu_{e^+}}{1+\mu_{e^+}}\frac{\varphi(e^-)}{\varphi(e^+)}+\frac{1}{1+\mu_{e^+}} \right)\\ & \cdot \mathbb{P}_{\boldsymbol{\lambda}, \boldsymbol{\mu}}\big(G(e) > \max\{G^*(e_1), G^*(e_2)\}\big),
\\
\mathbb{P}_{\boldsymbol{\lambda}, \boldsymbol{\mu}}(e_1, e_2 \in \mathcal{C}_{\rm CP }(\rho) \mid e \in \mathcal{C}_{\rm CP}(\rho ) , \mathcal{I}_3)&= \left( \frac{\mu_{e^+}}{2+\mu_{e^+}}\frac{\varphi(e^-)}{\varphi(e^+)}+\frac{2}{2+\mu_{e^+}} \right)\\
& \cdot \mathbb{P}_{\boldsymbol{\lambda}, \boldsymbol{\mu}}\big(G(e) > \max\{G^*(e_1), G^*(e_2)\}\big).
   \end{align*}
Using the law of total probability, we obtain \eqref{condp.e1e2}

\textbf{Step 2.} 
For each edge $g=(u^{-1},u)$, the quantity
$R(g) = \prod_{z \in \mathcal{P}_{u^{-1}}} \mu_z$ is seen as the resistance of $g$ associated to the process $\X$ in the non-excited regime. For each edge $g=\{u^{-1},u\} \in \mathcal{P}_e$, we define a modified resistance
\[
\widetilde{R}(g) := 2R(g) = 2 \prod_{z \in \mathcal{P}_{u^{-1}}} \mu_z,
\]
and for each edge $g \notin \mathcal{P}_e$ we keep  $\widetilde{R}(g)=R(g)$. Let $\widetilde{\boldsymbol\mu}=(\widetilde{\mu}_v)_{v\in V}$ be the bias parameters induced by $(\widetilde{R}(g))_{g\in E}$. In this step, we show that

\begin{align}\label{clock.inq}
     \mathbb{P}_{\boldsymbol{\lambda}, \boldsymbol{\mu}}\bigl(G(e) > \max\{G^{*}(e_1), G^{*}(e_2)\}\bigr) \le \mathbb{P}_{\boldsymbol{\lambda},\boldsymbol{\widetilde{\mu}}}({G}(e) > G^*(e_1)) \cdot \mathbb{P}_{\boldsymbol{\lambda},\boldsymbol{\widetilde{\mu}}}({G}(e) > G^*(e_2)).
\end{align}

In fact,
let $\widetilde{\X}$ be a ${\rm GOERW}(\boldsymbol{\lambda}, \boldsymbol{\widetilde{\mu}})$ process such that $\X$ and $\widetilde{\X}$ are constructed on the same probability space using the same collection of independent exponential clocks $\boldsymbol{\xi}$, with the difference that for edges on $\mathcal{P}_e$ the clock rates in the non-excited regime are halved according to the modified resistances. Define the quantity $\widetilde{G}(e)$ associated to $\widetilde{\X}$ in the same way as $G(e)$.
Note that $\widetilde{G}(e)$ is an exponential random variable with rate $p_e/2$. 
Hence
\begin{align}\label{tildeG}
\mathbb{P}_{\boldsymbol{\lambda},\boldsymbol{\widetilde{\mu}}}({G}(e) > G^*(e_j))=\mathbb{P}_{\boldsymbol{\lambda},\boldsymbol{\mu}}(\widetilde{G}(e) > G^*(e_j))=    \int_0^{+\infty} e^{-\frac{p_e}{2}x} f_j(x)\rmd x.
\end{align}    
Since $G(e)$ is independent of $G^{*}(e_1)$ and $G^{*}(e_2)$, 
we note that
\begin{align}\label{Gmax.inq}
    \mathbb{P}_{\boldsymbol{\lambda}, \boldsymbol{\mu}}\bigl(G(e) > \max\{G^{*}(e_1), G^{*}(e_2)\}\bigr)
 &= \int_{0}^{+\infty} \int_{0}^{+\infty}
e^{-p_e \max\{x_1, x_2\}} f_1(x_1) f_2(x_2) \rmd x_1 \rmd x_2\\
\nonumber & \le \int_0^{+\infty} \int_0^{+\infty} e^{-p_e \frac{x_1+ x_2}{2}} f_1(x_1) f_2(x_2) \rmd x_1 \rmd x_2 
\\
\nonumber &\le  \int_0^{+\infty} e^{-p_e \frac{x_1}{2}} f_1(x_1) \rmd x_1\cdot\int_0^{+\infty} e^{-p_e \frac{x_2}{2}} f_2(x_2) \rmd x_2.
\end{align}
where in the second step we use the elementary inequality $\max\{x_1, x_2\} \ge \frac{x_1 + x_2}{2}$. Combining \eqref{tildeG} and \eqref{Gmax.inq}, we obtain \eqref{clock.inq}.

 \textbf{Step 3.} In this step, we show that for $j\in \{0,1\}$, 
\begin{align}\label{coup.inq}\mathbb{P}_{\boldsymbol{\lambda},\boldsymbol{\widetilde{\mu}}}\big(e_j \in\mathcal{C}_{\rm CP}(\rho) \mid e \in\mathcal{C}_{\rm CP}(\rho)\big) \le \exp(K)\cdot  {\mathbb{P}_{\boldsymbol{\lambda},\boldsymbol{{\mu}}}\big(e_j \in\mathcal{C}_{\rm CP}(\rho) \mid e \in\mathcal{C}_{\rm CP}(\rho)\big)}
\end{align}
where $K:=2+\sup_{v\in V, |v|\ge 2}\frac{R(v^{-1},v)}{\varphi(v^{-1})}$.

 Define the function $\widetilde{\varphi}$ and $\widetilde{\psi}$ associated to $\widetilde\X$ in the same way as $\varphi$ and $\psi$. Note that for each vertex $u$ such that $e^+\le u$, we have
\[
\widetilde{\varphi}(u)=\sum_{g\in\mathcal{P}_u} \widetilde R(g)
=
\varphi(e^+) + \varphi(u).
\]
Using the fact that $\frac{1}{\mu_{u^{-1}} + 1} + \frac{\mu_{u^{-1}}}{\mu_{u^{-1}} + 1} \cdot \frac{\varphi(u^{-2})}{\varphi(u)}=\frac{\varphi(u^{-1})}{\varphi(u)}$, we deduce from \eqref{def.psi} that
\begin{align}\label{psi}
   {\psi}(u^{-1}, u) = \frac{1+ \lambda_{u^{-1}}\cdot \frac{{\varphi}(u^{-2})}{{\varphi}(u)} 
    +(\deg(u^{-1}) - 2) \cdot\frac{{\varphi}(u^{-1})}{{\varphi}(u)}}{\lambda_{u^{-1}} + \deg(u^{-1}) - 1}.
\end{align}
We also have
\begin{align}\label{psi.new}
     \widetilde{\psi}(u^{-1}, u) = \frac{1+ \lambda_{u^{-1}}\cdot \frac{\widetilde{\varphi}(u^{-2})}{\widetilde{\varphi}(u)} 
    +(\deg(u^{-1}) - 2) \cdot\frac{\widetilde{\varphi}(u^{-1})}{\widetilde{\varphi}(u)}}{\lambda_{u^{-1}} + \deg(u^{-1}) - 1}. 
\end{align}
Using Lemma \ref{phical} together with \eqref{psi}-\eqref{psi.new}, we thus obtain
\begin{align}\label{ratio.P}
 \nonumber \frac{\mathbb{P}_{\boldsymbol{\lambda},\boldsymbol{\widetilde{\mu}}}\big(e_j \in\mathcal{C}_{\rm CP}(\rho) \mid e \in\mathcal{C}_{\rm CP}(\rho)\big)}{\mathbb{P}_{\boldsymbol{\lambda},\boldsymbol{{\mu}}}\big(e_j \in\mathcal{C}_{\rm CP}(\rho) \mid e \in\mathcal{C}_{\rm CP}(\rho)\big)}
&= \frac{\frac{\mathbb{P}_{\boldsymbol{\lambda},\boldsymbol{\widetilde{\mu}}}\big(e_j \in\mathcal{C}_{\rm CP}(\rho)\big)}{\mathbb{P}_{\boldsymbol{\lambda},\boldsymbol{\widetilde{\mu}}}\big(e \in\mathcal{C}_{\rm CP}(\rho)\big)}}{\frac{\mathbb{P}_{\boldsymbol{\lambda},\boldsymbol{{\mu}}}\big(e_j \in\mathcal{C}_{\rm CP}(\rho) \big)}{\mathbb{P}_{\boldsymbol{\lambda},\boldsymbol{{\mu}}}\big(e \in\mathcal{C}_{\rm CP}(\rho) \big)}} =
\prod_{\{u^{-1},u\}\in \mathcal{P}_{e_j}\setminus\mathcal{P}_{e}}
\left(
\frac{\widetilde{\psi}(u^{-1},u)}
{\psi(u^{-1},u)}
\right)
\\
&= \prod_{\{u^{-1},u\}\in \mathcal{P}_{e_j}\setminus\mathcal{P}_{e}}\frac{1 + \lambda_{u^{-1}} \frac{\widetilde{\varphi}(u^{-2})}{\widetilde{\varphi}(u)} + (\text{deg}(u^{-1}) - 2) \frac{\widetilde{\varphi}(u^{-1})}{\widetilde{\varphi}(u)}}{1 + \lambda_{u^{-1}} \frac{\varphi(u^{-2})}{\varphi(u)} + (\text{deg}(u^{-1}) - 2) \frac{\varphi(u^{-1})}{\varphi(u)}}.
\end{align}
Using the inequality $1+x \le e^x$, we get
\begin{align}\label{exp.inq}
\nonumber  &\prod_{\{u^{-1},u\}\in \mathcal{P}_{e_j}\setminus\mathcal{P}_{e}}\frac{1 + \lambda_{u^{-1}} \frac{\widetilde{\varphi}(u^{-2})}{\widetilde{\varphi}(u)} + (\text{deg}(u^{-1}) - 2) \frac{\widetilde{\varphi}(u^{-1})}{\widetilde{\varphi}(u)}}{1 + \lambda_{u^{-1}} \frac{\varphi(u^{-2})}{\varphi(u)} + (\text{deg}(u^{-1}) - 2) \frac{\varphi(u^{-1})}{\varphi(u)}}\\
\nonumber &=   \prod_{\{u^{-1},u\}\in \mathcal{P}_{e_j}\setminus\mathcal{P}_{e}} \left( 1+\ \frac{\lambda_{u^{-1}} \left( \frac{\widetilde{\varphi}(u^{-2})}{\widetilde{\varphi}(u)} - \frac{\varphi(u^{-2})}{\varphi(u)} \right) + (\text{deg}(u^{-1}) - 2) \left( \frac{\widetilde{\varphi}(u^{-1})}{\widetilde{\varphi}(u)} - \frac{\varphi(u^{-1})}{\varphi(u)} \right)}{1 + \lambda_{u^{-1}} \frac{\varphi(u^{-2})}{\varphi(u)} + (\text{deg}(u^{-1}) - 2) \frac{\varphi(u^{-1})}{\varphi(u)}} \right) 
\\ 
&\le\exp \left( \sum_{\{u^{-1},u\}\in \mathcal{P}_{e_j}\setminus\mathcal{P}_{e}} \frac{\lambda_{u^{-1}} \left( \frac{\widetilde{\varphi}(u^{-2})}{\widetilde{\varphi}(u)} - \frac{\varphi(u^{-2})}{\varphi(u)} \right) + (\text{deg}(u^{-1}) - 2) \left( \frac{\widetilde{\varphi}(u^{-1})}{\widetilde{\varphi}(u)} - \frac{\varphi(u^{-1})}{\varphi(u)} \right)}{1 + \lambda_{u^{-1}} \frac{\varphi(u^{-2})}{\varphi(u)} + (\text{deg}(u^{-1}) - 2) \frac{\varphi(u^{-1})}{\varphi(u)}} \right).
\end{align}
Notice that
$$1 + \lambda_{u^{-1}} \frac{\varphi(u^{-2})}{\varphi(u)} + (\text{deg}(u^{-1}) - 2) \frac{\varphi(u^{-1})}{\varphi(u)}\ge  \frac{\varphi(u^{-2})}{\varphi(u)}(\lambda_{u^{-1}}+\deg(u^{-1})-1).$$
Moreover, 
\begin{align*}
    \frac{\widetilde{\varphi}(u^{-1})}{\widetilde{\varphi}(u)} - \frac{\varphi(u^{-1})}{\varphi(u)}
    &= \frac{\varphi(u^{-1}) + \varphi(e^+)}{\varphi(u) + \varphi(e^+)} - \frac{\varphi(u^{-1})}{\varphi(u)}=\frac{\varphi(e^+)R(u^{-1},u)}{\varphi(u)(\varphi(e^+)+\varphi(u))}.
\end{align*}
Similarly,
$$\frac{\widetilde{\varphi}(u^{-2})}{\widetilde{\varphi}(u)} - \frac{\varphi(u^{-2})}{\varphi(u)}
= \frac{\varphi(e^+)(R(u^{-2},u^{-1})+R(u^{-1},u))}{\varphi(u)(\varphi(e^+)+\varphi(u))}
.$$
It follows that
\begin{align*}
    &\lambda_{u^{-1}} \left( \frac{\widetilde{\varphi}(u^{-2})}{\widetilde{\varphi}(u)} - \frac{\varphi(u^{-2})}{\varphi(u)} \right) + (\text{deg}(u^{-1}) - 2) \left( \frac{\widetilde{\varphi}(u^{-1})}{\widetilde{\varphi}(u)} - \frac{\varphi(u^{-1})}{\varphi(u)} \right)\\
    & \le    \frac{\varphi(e^+)\big(R(u^{-2},u^{-1})+R(u^{-1},u)\big)}{\varphi(u)(\varphi(e^+)+\varphi(u))}.(\lambda_{u^{-1}}+\deg(u^{-1})-1).
\end{align*}
Hence
\begin{align}\label{sum.bound}
  \nonumber  &\sum_{\{u^{-1},u\}\in \mathcal{P}_{e_j}\setminus\mathcal{P}_{e}} \frac{\lambda_{u^{-1}} \left( \frac{\widetilde{\varphi}(u^{-2})}{\widetilde{\varphi}(u)} - \frac{\varphi(u^{-2})}{\varphi(u)} \right) + (\text{deg}(u^{-1}) - 2) \left( \frac{\widetilde{\varphi}(u^{-1})}{\widetilde{\varphi}(u)} - \frac{\varphi(u^{-1})}{\varphi(u)} \right)}{1 + \lambda_{u^{-1}} \frac{\varphi(u^{-2})}{\varphi(u)} + (\text{deg}(u^{-1}) - 2) \frac{\varphi(u^{-1})}{\varphi(u)}}\\
\nonumber & \le \varphi(e^+)\sum_{\{u^{-1},u\}\in \mathcal{P}_{e_j}\setminus\mathcal{P}_{e}} \frac{R(u^{-2},u^{-1})+R(u^{-1},u)}{\varphi(u^{-2})(\varphi(e^+)+\varphi(u))}\\
\nonumber & \le \varphi(e^+)\sum_{\{u^{-1},u\}\in \mathcal{P}_{e_i}\setminus\mathcal{P}_{e}} \frac{R(u^{-2},u^{-1})+R(u^{-1},u)}{\varphi(u^{-2})\varphi(u)}\\
\nonumber & = \varphi(e^+) \sum_{\{u^{-1},u\}\in \mathcal{P}_{e_j}\setminus\mathcal{P}_{e}}\left(  \frac{1}{\varphi(u^{-2})}-\frac{1}{\varphi(u)} \right) \\
 & = \varphi(e^+) \left( \frac{1}{\varphi(e^{-})}+\frac{1}{\varphi(e^+)}-\frac{1}{\varphi(e_j^{-})}-\frac{1}{\varphi(e_j^+)} \right) \le 2+\frac{R(e)}{\varphi(e^-)}\le K.
\end{align}
Combining \eqref{ratio.P}, \eqref{exp.inq} and \eqref{sum.bound}, we obtain \eqref{coup.inq}. 

\textbf{Step 4.} In this step, we finalize the proof using the results from the previous steps. 
From \eqref{coup.inq} and \eqref{condp.ej}, we have
\begin{align*}
    & {\mathbb{P}_{\boldsymbol{\lambda},\boldsymbol{{\mu}}}\big(e_1 \in\mathcal{C}_{\rm CP}(\rho) \mid e \in\mathcal{C}_{\rm CP}(\rho)\big)}      \cdot {\mathbb{P}_{\boldsymbol{\lambda},\boldsymbol{{\mu}}}\big(e_2 \in\mathcal{C}_{\rm CP}(\rho) \mid e \in\mathcal{C}_{\rm CP}(\rho)\big)}\\
    & \ge \exp(-2K)  \cdot \mathbb{P}_{\boldsymbol{\lambda},\boldsymbol{\widetilde{\mu}}}\big(e_1 \in\mathcal{C}_{\rm CP}(\rho) \mid e \in\mathcal{C}_{\rm CP}(\rho)\big) \cdot\mathbb{P}_{\boldsymbol{\lambda},\boldsymbol{\widetilde{\mu}}}\big(e_2 \in\mathcal{C}_{\rm CP}(\rho) \mid e \in\mathcal{C}_{\rm CP}(\rho)\big)\\
    & = \exp(-2K) \cdot q_e^2 \cdot  \mathbb{P}_{\boldsymbol{\lambda},\boldsymbol{\widetilde{\mu}}}\big( G(e)>G^*(e_1) \big)  \cdot \mathbb{P}_{\boldsymbol{\lambda},\boldsymbol{\widetilde{\mu}}}\big( G(e)>G^*(e_2) \big),
\end{align*}
where $$q_e:=1 - \frac{ \left(1 - \frac{\widetilde{\varphi}(e^-)}{\widetilde{\varphi}(e^+)}\right) \left( \lambda_{e^+} + (\deg(e^+)-2) \frac{\widetilde{\mu}_{e^+}}{\widetilde{\mu}_{e^+}+1} \right) }{ \lambda_{e^+} + \deg(e^+) - 1 }.$$
Using the above bound together with \eqref{clock.inq} and \eqref{condp.e1e2}, we obtain that 
\begin{align*}
   & {\mathbb{P}_{\boldsymbol{\lambda},\boldsymbol{{\mu}}}\big(e_1 \in\mathcal{C}_{\rm CP}(\rho) \mid e \in\mathcal{C}_{\rm CP}(\rho)\big)}      \cdot {\mathbb{P}_{\boldsymbol{\lambda},\boldsymbol{{\mu}}}\big(e_2 \in\mathcal{C}_{\rm CP}(\rho) \mid e \in\mathcal{C}_{\rm CP}(\rho)\big)}\\
    & \ge \exp(-2K) \cdot q_e^2 \cdot    \mathbb{P}_{\boldsymbol{\lambda}, \boldsymbol{\mu}}\bigl(G(e) > \max\{G^{*}(e_1), G^{*}(e_2)\}\bigr)\\    &  \ge \exp(-2K) \cdot\frac{q_e^2}{s_e} \cdot {\mathbb{P}_{\boldsymbol{\lambda},\boldsymbol{{\mu}}}\big(e_1, e_2 \in\mathcal{C}_{\rm CP}(\rho) \mid e \in\mathcal{C}_{\rm CP}(\rho)\big)},
\end{align*}
where
$$s_e:= 1 - \frac{ \left(1 - \frac{\varphi(e^-)}{\varphi(e^+)}\right) \left( \lambda_{e^+} + \frac{2\mu_{e^+}}{1+\mu_{e^+}} + \frac{(\deg(e^+)-3)\mu_{e^+}}{2+\mu_{e^+}}  \right) }{ \lambda_{e^+} + \deg(e^+) - 1 }.$$ 
Notice that
    $q_e\ge \frac{\varphi(e^-)}{\varphi(e^+)}\ge \frac{1}{1+R(e)/\varphi(e^-)}\ge \frac{1}{1+K}$
and that 
    $s_e\le 1$.
Hence
\begin{align*}
&{\mathbb{P}_{\boldsymbol{\lambda},\boldsymbol{{\mu}}}\big(e_1, e_2 \in\mathcal{C}_{\rm CP}(\rho) \mid e \in\mathcal{C}_{\rm CP}(\rho)\big)}\\
& \le M \cdot {\mathbb{P}_{\boldsymbol{\lambda},\boldsymbol{{\mu}}}\big(e_1 \in\mathcal{C}_{\rm CP}(\rho) \mid e \in\mathcal{C}_{\rm CP}(\rho)\big)}      \cdot {\mathbb{P}_{\boldsymbol{\lambda},\boldsymbol{{\mu}}}\big(e_2 \in\mathcal{C}_{\rm CP}(\rho) \mid e \in\mathcal{C}_{\rm CP}(\rho)\big)}.
\end{align*}
with $M=(1+K)^2\exp(2K)$. Hence the ruin percolation is quasi-independent.
\end{proof}

\begin{remark}\label{rem2}
The proof of quasi-independence for the percolation associated with the homogeneous OERW in Section 9.1 of \cite{H2018} appears to be incomplete, with several steps requiring further justification. In particular, formula (9.28) in \cite{H2018} does not appear to be accurate as stated and can be corrected using \eqref{ratio.P} from our proof. Moreover, our correlation inequality for the percolation is more general and stronger than the one in \cite{H2018}, and our result does not require the tree to have bounded degree.
\end{remark}

\subsection{Criteria of transience and recurrence for GOERW}\label{sec:RT}

The results of Section \ref{sec:quasi} allow one to prove Theorem \ref{RT.thm} by adapting the argument for once-reinforced random walks from Section 8 in \cite{CKS2020}. In this section, we instead provide a shorter and more direct proof.

Recall from Proposition \ref{phical} that $\P_{\boldsymbol{\lambda}, \boldsymbol{\mu}}(\rho\leftrightarrow e^+)=\Psi(e)$ and 
$$\P_{\boldsymbol{\lambda}, \boldsymbol{\mu}}(e \text{ is closed} \mid \rho\leftrightarrow e^-)=\frac{\P_{\boldsymbol{\lambda}, \boldsymbol{\mu}}( \rho\leftrightarrow e^-)-\P_{\boldsymbol{\lambda}, \boldsymbol{\mu}}(\rho\leftrightarrow e^+)}{\P_{\boldsymbol{\lambda}, \boldsymbol{\mu}}( \rho\leftrightarrow e^-)}=1-\psi(e).$$
Hence 
$$c(e)=1 \text{ for $|e|=1$} \text{ and  \ \ } c(e)=\frac{1}{1 - \psi(e)} \Psi(e) \text{ for $|e|\ge 2$} $$ are the adapted conductances to the ruin percolation according to the definition \eqref{adt.c}.
Recall that
\[
RT(\Tcal, \boldsymbol{\lambda},\boldsymbol{\mu}) = \sup \left\{ \gamma > 0 : \inf_{\pi \in \Pi} \sum_{e \in \pi} (\Psi(e))^{\gamma} > 0 \right\}.
\]

    \begin{proof}[Proof of \Cref{RT.thm}]
\textbf{Case 1: $RT(\Tcal, \boldsymbol{\lambda},\boldsymbol{\mu})> 1$}.\\ Choose $\gamma \in \big(1, RT(\Tcal, \boldsymbol{\lambda},\boldsymbol{\mu}) \big)$.
By definition, we have
$\inf_{\pi\in \Pi} \sum_{e\in \pi} \big(\Psi(e)\big)^{\gamma}>0.$
By the max-flow min-cut Theorem, there exists a non-zero flow $(\theta(e))_{e\in E}$ such that
$\theta(e)\le \big(\Psi(e)\big)^{\gamma}$.
Hence $$\sum_{e\in\mathcal{P}_v}\frac{\theta(e)}{c(e)}\le\sum_{e\in \mathcal{P}_v} (1-\psi(e))\prod_{g\le e}(\psi(g))^{\gamma-1}.$$
By Proposition 8.5 in \cite{CKS2020}, the left-hand side is uniformly bounded for $v\in V$. Hence $\sup_{v\in V}\sum_{e\in\mathcal{P}_v}\frac{\theta(e)}{c(e)}<\infty$.
We note that (see, e.g., Proposition 16.1 in \cite{LP2016}) $$\mathcal{E}(\theta):=\sum_{e\in E}\frac{\theta(e)^2}{c(e)}=\int_{\partial \Tcal} V_{\theta}(\xi) \rmd m_{\theta}(\xi)$$
where $V_{\theta}(\xi):=\sum_{e\in \xi}\frac{\theta(e)}{c(e)}$ for $\xi\in\partial \Tcal$, and $m_{\theta}$ is the harmonic measure induced by flow $\theta$ on  $\partial \Tcal$. 
As $V_{\theta}$ is  bounded in ${\partial}\Tcal$, we have $\mathcal{E}(\theta)<\infty$. On the other hand, by Lemma \ref{per1.lem}, the ruin percolation is quasi-independent. Hence, by Proposition \ref{lyons1989}(ii), we have $$\P_{\boldsymbol{\lambda}, \boldsymbol{\mu}}(\rho\leftrightarrow\infty)=\P_{\boldsymbol{\lambda}, \boldsymbol{\mu}}(|\mathcal{C}_{\rm CP}(\rho)|=\infty)>0.$$ 
Combining this fact with Lemma \ref{per2.lem}, we conclude that the process $\X$ is transient. \\

\textbf{Case 2: $RT(\Tcal, \boldsymbol{\lambda},\boldsymbol{\mu}) < 1$}.\\
By definition, we have
\[
\inf_{\pi \in \Pi} \sum_{e \in \pi} \mathbb{P}_{\boldsymbol{\lambda}, \boldsymbol{\mu}}(\rho \leftrightarrow e^+)
= \inf_{\pi \in \Pi} \sum_{e \in \pi} \Psi(e)
= 0.
\]
Using Proposition \ref{lyons1989}(i), we conclude that
$\mathbb{P}_{\boldsymbol{\lambda}, \boldsymbol{\mu}}(\rho \leftrightarrow \infty) = \P_{\boldsymbol{\lambda}, \boldsymbol{\mu}}(|\mathcal{C}_{\rm CP}(\rho)|=\infty)=0.$ Hence, 
$\mathbf{X}$ is recurrent.
\end{proof}

    \section{Once-excited random walks in random environment}\label{sec:main}

    Throughout this section we assume that for all $v\in V\setminus\{\rho\}$,
    $$\mu_{v}=1 \text{ and }\lambda_v=1+\alpha_v \deg(v),$$
    where $(\alpha_v)_{v\in V\setminus\{\rho\}}$ are i.i.d. non-negative random variables with $m=\E[1/(\alpha_v+1)].$ We will show that $\Psi(e)$ concentrates around $|e|^{-2+m\pm\eps}$ with high probability. Using the Borel-Cantelli lemma, we infer that there is a sequence of cutsets $(\pi_n)_{n\ge0}$ such that $\Psi(e)\le |e|^{-2+m+\eps}$ a.s. for each edge $e\in \bigcup_{n\ge0} \pi_{n}$. Combining this fact and Theorem \ref{RT.thm}, we obtain the recurrence case of Theorem \ref{main.thm}. In order to demonstrate the transience case of Theorem \ref{main.thm}, we use the quasi-independent percolation technique to show that, with positive probability, there is an infinite subtree $\widetilde{\Tcal}$ containing the root such that $\Psi(e)\ge \kappa^{-1} |e|^{-2+m-\eps}$, with $\kappa>0$, for each edge $e$ of $\widetilde{\Tcal}$ and $\text{br}_{r}(\widetilde{\Tcal})>\text{br}_{r}(\Tcal)-2\eps$. This fact, together with Theorem \ref{RT.thm}, yields the transience case.

    Set $\kappa:= 1+\max_{v\sim \rho} \alpha_v$.
    Fix $\eps>0$ and
    let $$\mathcal{B}_e:=\Big\{\kappa^{-1}|e|^{-2+m-\eps} \le \Psi(e)\le |e|^{-2+m+\eps}\Big\}.$$
   \begin{lemma}\label{lem.B} We have
    $$ \P( \mathcal{B}_e^c)\le  K_1 \exp(- K_2(\log |e|)^2),$$
    where $K_1$ and $K_2$ are some positive constants depending only on $\eps$.
    \end{lemma}
    \begin{proof}
    Substituting 
         $\mu_{v}=1$ for all $v\in V\setminus\{\rho\}$ into \eqref{def.psi}, we have
    $$\psi(v^{-1},v)=1-\frac{2(\lambda_{v^{-1}}-1)+\deg({v^{-1}})}{|v|(\lambda_{v^{-1}}-1+\deg({v^{-1}}))} \text{ for $|v|\ge 2$.}$$
    As $\lambda_v=1+\alpha_v \deg(v)$, we obtain
    $$\psi(v^{-1},v)=1-\frac{1}{|v|}\frac{2\alpha_{v^{-1}}+1}{\alpha_{{v^{-1}}}+1},$$
  {for $|v|\ge 2$ and $\psi(\rho,v)=1$ for $v\sim \rho$.}
    Define $C_v:=\frac{2\alpha_{v^{-1}}+1}{\alpha_{v^{-1}}+1}$. We obtain
    $$\Psi(e)=\prod_{g \in \Pcal_e} \psi(g) = \frac{1}{2(1+\alpha_{v_1(e)})} \prod_{v\in\mathcal{P}_e, |v|\ge 3} \left( 1 - \frac{C_v}{|v|} \right),$$
   where $v_1(e)$ is the unique vertex on $\mathcal{P}_e$ such that $v_1(e)\sim \rho$. Hence
    $$\log(\Psi(e)) = \sum_{v\in\mathcal{P}_e, |v|\ge 3} \log\Big(1 - \frac{C_v}{|v|}\Big)- \log\left(2(1+\alpha_{v_1(e)})\right).$$
    Using the inequality $\log(1 - x) \le -x$ for $x \in (0, 1)$ with 
    $x=\frac{C_v}{|v|} \in (0,1) \text{ for } |v| \ge 3$, we obtain
    $$\log(\Psi(e)) \le -\sum_{v\in\mathcal{P}_e, |v|\ge 3} \frac{C_v}{|v|}.$$
    Using the inequality $\log(1 - x) \ge -x - \frac{x^2}{1-x}$ with $x=\frac{C_v}{|v|} \in (0,1)$ for $|v| \ge 3$, we get
    $$
    \log(\Psi(e)) \ge -\sum_{v\in\mathcal{P}_e, |v|\ge 3} \frac{C_v}{|v|} -   \sum_{v\in\mathcal{P}_e, |v|\ge 3}\frac{(\frac{C_v}{|v|})^2}{1-\frac{C_v}{|v|}}- \log\left(2\kappa\right),
    $$
    where we recall that $\kappa= 1+\max_{v\sim \rho} \alpha_v$.
  The second sum in the right hand side is bounded by a constant. In fact, since $C_v \in [1,2]$ and $|v| \ge 3$, we have
    $$
    \sum_{v\in\mathcal{P}_e, |v|\ge 3}\frac{(\frac{C_v}{|v|})^2}{1-\frac{C_v}{|v|}} \le \sum_{v\in\mathcal{P}_e, |v|\ge 3} \frac{\frac{4}{|v|^2}}{1-\frac{2}{|v|}}
    \le  \sum_{v\in\mathcal{P}_e, |v|\ge 3} \frac{\frac{4}{|v|^2}}{\frac{1}{3}}
    = 12\sum_{n=3}^{|e|} \frac{1}{n^2}\le 12\sum_{n=1}^{\infty}\frac{1}{n^2}=2\pi^2.
    $$
    Therefore
    \begin{align}\label{R.E1}
    - S_e -  2\pi^2-\log\left(2\kappa\right) \le  \log(\Psi(e)) \le -S_e,
    \end{align}
in which we set $$S_e:=\sum_{v\in\mathcal{P}_e, |v|\ge 3}\frac{C_v}{|v|}.$$ 
Applying Hoeffding's inequality (see Proposition \ref{Hoeffding}) to the random variables $\frac{C_v}{|v|}  \in \big[\frac{1}{|v|},\frac{2}{|v|}\big]$ for $v\in \mathcal{P}_e$ with $|v|\ge 3$, we have
    $$\P \big(|S_e-\E[S_e]|\geq t\big)\leq 2\exp \left(-{\frac {2t^{2}}{\sum_{v\in\mathcal{P}_e, |v|\ge 3}\frac{1}{|v|^2}}}\right).$$
Choosing $t=\frac{\eps}{2}\log|e|$ and using the fact that
    $$\sum_{v\in\mathcal{P}_e, |v|\ge 3}\frac{1}{|v|^2} \le \sum_{k=1}^{\infty}\frac{1}{k^2}=\frac{\pi^2}{6},$$
we obtain
    \begin{align}\label{Hoeffding1}
       \P \big(|S_e-\E[S_e]|\geq \frac{\eps}{2}\log|e| \big)\leq 2\exp\left(
    -\frac{3\varepsilon^{2}}{\pi^{2}}(\log|e|)^{2} \right).
    \end{align}
On the other hand, we notice that
    $$\E[S_e]=\sum_{v\in\mathcal{P}_e, |v|\ge 3}\frac{\E[C_v]}{|v|}=\sum_{v\in\mathcal{P}_e, |v|\ge 3}\frac{\E[2-\frac{1}{\alpha_{v^{-1}}+1}]}{|v|}=(2-m)\sum_{k=3}^{|e|}\frac{1}{k}.$$
    Using the harmonic series bound $\log(n) \le \sum_{k=1}^n \frac{1}{k} \le 1+\log(n)$, we have
    \begin{align}\label{Ex1}
        (2-m) \left(\log|e|-\frac{3}{2}\right) \le \E[S_e] \le (2-m) \log|e|.
    \end{align}
It follows from (\ref{R.E1}) and (\ref{Ex1}) that, on the event $\{\Psi(e) < \kappa^{-1}|e|^{-2+m-\eps}\}$, we have
\begin{equation}\label{lb}
    \begin{aligned}
        S_e-\E[S_e] & \ge (2-m+\eps)\log|e|-2\pi^2- \log(2) -(2-m) \log|e|\\
        &=\eps \log|e|-2\pi^2- \log(2).
    \end{aligned}
\end{equation}
   Similarly, it follows from (\ref{R.E1}) and (\ref{Ex1}) that, on the event $\{\Psi(e) > |e|^{-2+m+\eps}\}$, we have
   \begin{align}\label{ub}
       S_e-\E[S_e] \le (2-m-\eps)\log|e|-(2-m) \left(\log|e|-\frac{3}{2}\right)\le -\eps \log|e|+3.
   \end{align}
Combining (\ref{lb}) and (\ref{ub}), we obtain that
\begin{align*}
    \mathcal{B}_e^{c}& =  \big\{\Psi(e) < \kappa^{-1} |e|^{-2+m-\eps}\big\}\cup\{ \Psi(e) > |e|^{-2+m+\eps}\big\}
    \\ & \subset \{  S_e-\E[S_e] \ge  \eps \log|e|-2\pi^2- \log(2)\}\cup \{  S_e-\E[S_e] \le -\eps \log|e|+3\}\\
    & \subset \Big\{|S_e-\E[S_e]|\geq \frac{\eps}{2}\log|e|\Big\}
\end{align*} 
for large enough $|e|$.
Using the above inclusion and \eqref{Hoeffding1}, we deduce that
    $$
    \begin{aligned}
        \P(\mathcal{B}_e^{c}) 
        & \le \P \Big(|S_e-\E[S_e]|\geq \frac{\eps}{2}\log|e| \Big) \le 2\exp\left(
    -\frac{3\varepsilon^{2}}{\pi^{2}}(\log|e|)^{2} \right)
    \end{aligned}
$$
for large enough $|e|$. This ends the proof of the lemma. 
    \end{proof}
\begin{lemma}\label{lem.Psi.bound.new}
    For each $\eps>0$ and $b>{\rm br}_r(\Tcal)$, there exists a.s. a sequence of cutsets $(\pi_n)_{n\ge 0}$ such that $$\sum_{e\in\pi_n}|e|^{-b}<\exp(-n), \text{ for each } n\ge 0$$ and 
 $$\kappa^{-1}|e|^{-2+m-\eps} \le \Psi(e)\le |e|^{-2+m+\eps} \text{ for all } e\in \bigcup_{n\ge0} \pi_n.$$
\end{lemma}
\begin{proof}

    Recall that the  branching-ruin number of $\Tcal$ is given by
    $$\text{br}_r(\mathcal{T}) = \sup\left\{\gamma > 0: \inf_{\pi\in \Pi}\sum_{e\in\pi}|e|^{-\gamma} > 0\right\}.$$
    Fix $b>\text{br}_r(\mathcal{T})$, then 
    $$\inf_{\pi\in \Pi}\sum_{e\in\pi}|e|^{-b}=0.$$
    In particular, for each $n$ there exists a cutset $\pi_n$ such that
    $$\sum_{e\in\pi_n}|e|^{-b}<\exp(-n).$$
    In virtue of Lemma \ref{lem.B}, we have
    $$ \P( \mathcal{B}_e^c)\le K_1|e|^{-K_2\log|e|} < |e|^{-b},$$
    for all $|e| \ge N$ with some large enough integer $N$. Let 
    $$\mathcal{V}_n= \bigcap_{e\in \pi_n } \mathcal{B}_e.$$
    Notice that for $n\ge N$,
    $$\P(\mathcal{V}_n^c)=\P\left(\bigcup_{e\in \pi_n}\mathcal{B}_e^c\right)\le \sum_{e\in \pi_n}\P(\mathcal{B}_e^c) < \sum_{e\in \pi_n}|e|^{-b}<\exp(-n).$$
    Hence
    $$\sum_{n=N}^{\infty}\P(\mathcal{V}_n^c) < \sum_{n=N}^{\infty}\exp(-n)< \infty. $$
    By the first Borel–Cantelli lemma, we infer that the events $(\mathcal{V}_n^c)_{n\ge 1}$ occur finitely often almost surely.
    Hence, almost surely there exists a large enough $n_0 \in \N$ such that, the event $\mathcal{V}_n$ holds for all $n \ge n_0$.
    Equivalently, the event $\mathcal{B}_e$ holds for all edges $e\in \bigcup_{n\ge n_0}\pi_{n}$.
\end{proof}

    \begin{proof}[Proof of \Cref{main.thm}]

We first verify the conditions of Theorem \ref{RT.thm} for the case when $\mu_{v}=1$ for all $v\in V$ and $\lambda_v=1+\alpha_v \deg(v),$
    where $(\alpha_v)_{v\in V}$ are i.i.d. non-negative random variables. 
As $\mu_v = 1$ for all $v$, we have $R(v^{-1},v)=\prod_{\{u^{-1}, u\}\in \mathcal{P}_{v^{-1}}} \mu_{u}=1$ and thus $$\varphi(x) = \sum_{e \in \mathcal{P}_x} R(e)  = |x|.$$
  Furthermore
         $$\begin{aligned}\sup_{v\in V, |v|\ge 2} \frac{R(v^{-1},v)}{\varphi(v^{-1})}=\sup_{v\in V, |v|\ge 2}\frac{1}{|v|-1}=1.
\end{aligned}$$

We  consider the following two distinguished cases:

\textbf{a. When ${\rm br}_r(\Tcal)<2-m$.}

    Fix $\gamma\ge 1$.
    Let $\eps>0$ be sufficiently small such that $2-m-\eps >\text{br}_r(\mathcal{T})$. Let $b=(2-m-\eps)\gamma$. Note that $b> \text{br}_r(\mathcal{T}).$
    In virtue of Lemma \ref{lem.Psi.bound.new}, there exists a.s. a sequence of cutsets $(\pi_n)_{n\ge 0}$ such that
    $$\sum_{e\in\pi_n}|e|^{-(2-m-\eps)\gamma}<\exp(-n) \text{ for each  $n\ge0$, and  }
    \Psi(e)\le |e|^{-2+m+\eps} \text{ for all $e\in \bigcup_{n\ge0} \pi_n$.}$$
   It follows that a.s.
    $$\sum_{e \in \pi_n}\Psi(e)^{\gamma}\le \sum_{e \in \pi_n}|e|^{-(2-m-\eps)\gamma}<\exp(-n).$$
We thus have
    $$\inf_{\pi \in \Pi} \sum_{e \in \pi}\Psi(e)^{\gamma} \le \inf_{n \ge 0} \sum_{e \in \pi_n}\Psi(e)^{\gamma} = 0,$$
 {for all $\gamma\ge 1$}. Consequently, by the definition of $RT(\Tcal,\boldsymbol{\lambda},\boldsymbol{\mu})$, we must have that a.s. $RT(\Tcal, \boldsymbol{\lambda},\boldsymbol{\mu})<1.$
    Using Theorem \ref{RT.thm}, we infer that, under the annealed measure $\P$, the walk is a.s. recurrent.

     \textbf{ b. When ${\rm br}_r(\Tcal)>2-m$.}
     

    Fix $\varepsilon>0$. We define a percolation on $\Tcal$ by declaring an edge 
    $e$ to be open if the event $$\mathcal{B}_e=\Big\{\kappa^{-1}|e|^{-2+m-\eps} \le \Psi(e)\le |e|^{-2+m+\eps}\Big\}$$ holds, and closed otherwise. Removing all closed edges, the open edges form connected clusters. We denote by $\widetilde{\mathcal{T}}$ the cluster containing the root $\rho$, which is clearly a tree. We first prove that there exists a constant $p_0\in (0,1)$ such that for all edge $e=\{e^-, e^+\}$ with $e^-\le e^+$,
  \begin{align}\label{open.path}
      \P(\rho \leftrightarrow e^+)\ge p_0. 
  \end{align}
Indeed, fix a sufficiently large positive integer $n_0$. For each $e$ with $|e|\ge n_0$, choose $e_0\in\mathcal{P}_e$ with $|e_0|=n_0$. We notice that
    $$\P(\rho\leftrightarrow e^+)=\P(\rho\leftrightarrow e^+, \rho\leftrightarrow e^+_0)=\P(\rho\leftrightarrow e^+_0)\P(\rho\leftrightarrow e^+ \mid \rho\leftrightarrow e^+_0). $$
    Note that
    $$\P(\rho\leftrightarrow e^+ \mid \rho\leftrightarrow e^+_0)=\P\left(\bigcap_{g\in \mathcal{P}_{e}\setminus \mathcal{P}_{e_0}} \mathcal{B}_g\right)=1-\P\left(\bigcup_{g\in \mathcal{P}_{e}\setminus \mathcal{P}_{e_0}} \mathcal{B}_g^c\right).$$
   Using the union bound and Lemma \ref{lem.B}, we have 
    $$\P\left(\bigcup_{g\in \mathcal{P}_{e}\setminus \mathcal{P}_{e_0}} \mathcal{B}_g^c\right) \le \sum_{g\in \mathcal{P}_{e}\setminus \mathcal{P}_{e_0}}\P(\mathcal{B}_g^c)\le K_1\sum_{n\ge n_0} \exp(-K_2\log(n)^2).$$
    Therefore
    $$\P(\rho\leftrightarrow e^+ \mid \rho\leftrightarrow e^+_0) \ge 1 - K_1\sum_{n\ge n_0} \exp(-K_2\log(n)^2).$$
Since $\sum_{n\ge n_0} \exp(-K_2\log(n)^2)\to 0$ as $n_0\to \infty$. We choose $n_0$ sufficiently large such that
    $$\sum_{n\ge n_0} \exp(-K_2\log(n)^2) \le \frac{1}{2K_1}.$$
Hence 
    \begin{align*}
        \P(\rho\leftrightarrow e^+)\ge \frac{1}{2}\P(\rho\leftrightarrow e^+_0)\ge \frac{1}{2}\min_{ |e_0|=n_0}\P(\rho\leftrightarrow e_0^+),
    \end{align*}
    for all edge $e\in E$ such that $|e|\ge n_0$. This verifies \eqref{open.path}. 

We next prove that the above-mentioned percolation is quasi-independent. By \eqref{open.path}, we notice that for all vertices $x$ and $y$, 
    $$\mathbb{P}(\rho\leftrightarrow x \mid \rho\leftrightarrow x\wedge y ) = \frac{\mathbb{P}(\rho\leftrightarrow x, \rho\leftrightarrow x\wedge y)}{\mathbb{P}(\rho\leftrightarrow x\wedge y)}= \frac{\mathbb{P}(\rho\leftrightarrow x)}{\mathbb{P}(\rho\leftrightarrow x\wedge y)}\ge p_0.$$
    Similarly, we obtain that
    $$\mathbb{P}(\rho\leftrightarrow y \mid \rho\leftrightarrow x\wedge y ) \ge p_0.$$
   Therefore,
    $$\mathbb{P}(\rho\leftrightarrow x, \rho\leftrightarrow y \ |\ \rho\leftrightarrow x\wedge y ) \le 1\le \frac{1}{p_0^2}\mathbb{P}(\rho\leftrightarrow x \ |\ \rho\leftrightarrow x\wedge y )\mathbb{P}(\rho\leftrightarrow y \ |\ \rho\leftrightarrow x\wedge y ).$$
    Hence, the percolation is quasi-independent.
    In virtue of Lemma \ref{lem.B},
    $$\P(e \text{ is closed}) \le K_1 \exp(-K_2(\log|e|)^2).$$
    It follows that
    $$\P(e \text{ is open})\ge 1-\frac{\eps}{|e|},$$
    for all $|e|\ge n_1$, where $n_1$ is some sufficiently large positive integer. Using Proposition 4.4. in \cite{CHK2019} with the fact that the percolation is quasi-independent and  $\P(e \text{ is open}) \ge 1-\frac{\eps}{|e|}$ for all $|e|\ge n_1$, we deduce that
   \begin{align}\label{brr.inq}
      \P\big( \text{$\widetilde{\mathcal{T}}$ is infinite and $\text{br}_r(\widetilde{\Tcal})>\text{br}_r({\Tcal})-2\eps$}\big)>0.
   \end{align}
   
Assuming $\widetilde{\Tcal}$ is infinite, we denote by $\Pi_{\widetilde{\mathcal{T}}}$ the set of all cutsets of $\widetilde{\mathcal{T}}$.
       For each $\pi\in\Pi_T$, there is a cutset $\widetilde\pi\in\Pi_{\widetilde{\mathcal{T}}}$ where $\widetilde\pi$ consists only of open edges $\{v^{-1},v\}$ of $\pi$ with $v\in \widetilde{\Tcal}$. 
     For $\gamma>0$, we notice that $$\sum_{e \in \pi}\Psi(e)^{\gamma} \ge \sum_{e \in \widetilde{\pi}}\Psi(e)^{\gamma}.$$
    Hence, on the event $\widetilde{\mathcal{T}}$ is infinite, we have that for each $\gamma>0$,
    $$\inf_{\pi \in \Pi} \sum_{e \in \pi}\Psi(e)^{\gamma} \ge \inf_{\widetilde\pi \in \Pi_{\widetilde{\mathcal{T}}}} \sum_{e \in \widetilde{\pi}}\Psi(e)^{\gamma} \ge \kappa^{-1}\inf_{\widetilde\pi \in \Pi_{\widetilde{\mathcal{T}}}}\sum_{e \in \widetilde{\pi}} |e|^{-(2-m+\eps)\gamma},$$
    where in the second inequality we use the fact that $\mathcal{B}_e$ holds for all $e\in \widetilde{\mathcal{T}}$. As $2-m<\text{br}_r(\Tcal)$  we choose $\eps\in (0,1)$ sufficiently small such that $2-m+3\eps< \text{br}_r(\Tcal)$.
In virtue of \eqref{brr.inq}, we thus have that, with positive probability,  $2-m+\eps< \text{br}_r(\widetilde{\Tcal})$. Hence, there exists a $\gamma>1$ such that $(2-m+\eps)\gamma\le \text{br}_r(\widetilde{\Tcal})$ with positive probability.  
Therefore, by the definition of the branching-ruin number,
    $$\inf_{\pi \in \Pi} \sum_{e \in \pi}\Psi(e)^{\gamma} > 0,$$ for some $\gamma>1$
    with positive probability.
  As $\gamma>1$, by the definition of $RT(\Tcal,\boldsymbol{\lambda},\boldsymbol{\mu})$, we must have that
  $$\P\big(RT(\Tcal, \boldsymbol{\lambda},\boldsymbol{\mu})>1\big)>0.$$ 
    
    Using Theorem \ref{RT.thm}, we infer that under the annealed measure $\P$, the process $\X$ is transient  with positive probability. By Lemma \eqref{event.equiv}, under the quenched measure $\P_{\boldsymbol{\lambda}, \boldsymbol{\mu}}$, the process is either a.s. transient or a.s. recurrent. This yields the same zero-one law for the annealed measure $\P$. Hence, the annealed measure $\P$, the process is a.s. transient. 
      \end{proof}
\section{Appendix}
\subsection{Gambler's ruin}\label{sec:gam}

The Gambler's Ruin problem models a gambler betting repeatedly until either losing all money (ruin) or reaching a target fortune. Specifically:
\begin{itemize}
    \item The gambler starts with $i$ units of money ($0 < i < N$), and the opponent (e.g., a casino) has $N - i$ units, with total capital fixed at $N$.
    \item In each round, the gambler bets 1 unit. If at state $i$ (fortune = $i$), they win 1 unit (move to $i+1$) with probability $p_i$ or lose 1 unit (move to $i-1$) with probability $q_i = 1 - p_i$, where $p_i$ depends on $i$. Assume that $0<p_i<1$ for $1\le i\le N-1$.
    \item The game ends when the gambler reaches state 0 (ruin, no money) or state $N$ (success, wins all capital). 
    \item States 0 and $N$ are absorbing: once reached, the gambler stays there.
\end{itemize}
Let $S_n$ be the gambler's fortune after $n$ steps.
This is a discrete-time Markov chain on states $\{0, 1, \dots, N\}$ with transition probabilities:
\[
\P(S_{n+1} = j \mid S_n = i) =
\begin{cases}
p_i & \text{if } j = i + 1, \\
q_i = 1 - p_i & \text{if } j = i - 1, \\
0 & \text{otherwise},
\end{cases}
\]
for $i = 1, \dots, N-1$, and $\P(S_{n+1} = i \mid S_n = i) = 1$ for $i = 0, N$. 

Denote 
$\P_i(\cdot)=\P(\cdot\mid S_0=i)$.
Let $\tau_k=\inf\{n\ge 0: S_n=k\}$ be the first hitting time to state $k$. Let $x_i$ denote the probability of eventual ruin (reaching 0 before $N$) starting from state $i$, i.e.,
$$x_i=\P_i(\tau_0<\tau_N).$$ 
For $1\le i \le N-1$, the ratio
$$\mu_i:=\frac{q_i}{p_i}$$
is called the \textbf{bias} (toward $0$) at state $i$.
\vspace{5pt}

\begin{proposition}\label{gambler.ruin} For $0\le i\le N$, we have 
$$x_i=1-\frac{\varphi(i)}{\varphi(N)},
$$
where 
\[
\varphi(k) := \sum_{j=1}^k \prod_{h=1}^{j-1} \mu_h, \quad k = 1, \dots, N,
\]
with  conventions $\varphi(0) = 0$.

\end{proposition}
\vspace{5pt}

The above result follows immediately from the fact that $(x_i)_{0\le i\le N}$ is the solution to the difference equation 
$$x_i=q_ix_{i-1}+p_ix_{i+1}, \text{ for $1\le i\le N-1$}  $$
with the boundary conditions $x_0=1$ and $x_N=0$.

\subsection{Hoeffding's inequality}

\begin{proposition}[Hoeffding's inequality]\label{Hoeffding}
Let $(X_n)_{n\ge 1}$ be independent random variables such that 
$a_n \leq X_{n}\leq b_n$ almost surely. 
Let $$S_{n}=X_{1}+\cdots +X_{n}.$$
For all $t > 0$,
\begin{align*}
   & \P \big(|S_{n}-\E[S_n]|\geq t\big)\leq 2\exp \left(-{\frac {2t^{2}}{\sum_{i=1}^n(b_i-a_i)^{2}}}\right).
\end{align*} 
\end{proposition}
\vspace{5pt}

See, e.g., Theorem 2.16 in \cite{BDR2015}.

\section*{Acknowledgment} 
The authors would like to thank Andrea Collevecchio and Maximilian Nitzschner for their valuable comments and feedback, which helped improve the quality of the manuscript.
\bibliography{refs}  
\bibliographystyle{amsplain}
\end{document}